\documentclass[10pt]{article}

\usepackage{mathrsfs}
\DeclareMathAlphabet{\mathpzc}{OT1}{pzc}{m}{it}

\usepackage{color}
\usepackage{amsmath}
\usepackage{graphicx}
\usepackage{amsfonts}
\usepackage{amssymb}%
\usepackage{latexsym}
\usepackage{psfrag}

\newtheorem{theorem}{Theorem}[section]
\newtheorem{corollary}[theorem]{Corollary}
\newtheorem{definition}{Definition}[section]
\newenvironment{proof}[1][Proof]{\noindent \emph{#1.} }
{\hfill \ \rule{0.5em}{0.5em}}
\newtheorem{lemma}{Lemma}[section]
\newtheorem{proposition}[theorem]{Proposition}
\newtheorem{problem}{Problem}[section]

\newtheorem{example}[theorem]{Example}
\newtheorem{remark}[theorem]{Remark}

\newcounter{mycount}

\numberwithin{equation}{section}

\usepackage[a4paper]{geometry}
\newcommand{\beq}{\begin{equation}}
\newcommand{\eeq}{\end{equation}}
\newcommand{\beqs}{\begin{equation*}}
\newcommand{\eeqs}{\end{equation*}}
\newcommand{\bit}{\begin{itemize}}
\newcommand{\eit}{\end{itemize}}
\newcommand{\ben}{\begin{enumerate}}
\newcommand{\een}{\end{enumerate}}
\newcommand{\bal}{\begin{align}}
\newcommand{\eal}{\end{align}}
\newcommand{\bals}{\begin{align*}}
\newcommand{\eals}{\end{align*}}
\newcommand{\bse}{\begin{subequations}}
\newcommand{\ese}{\end{subequations}}
\newcommand{\bpr}{\begin{proposition}}
\newcommand{\epr}{\end{proposition}}
\newcommand{\bre}{\begin{remark}}
\newcommand{\ere}{\end{remark}}
\newcommand{\bpf}{\begin{proof}}
\newcommand{\epf}{\end{proof}}
\newcommand{\ble}{\begin{lemma}}
\newcommand{\ele}{\end{lemma}}
\newcommand{\bco}{\begin{corollary}}
\newcommand{\eco}{\end{corollary}}
\newcommand{\bex}{\begin{example}}
\newcommand{\eex}{\end{example}}
\newcommand{\bth}{\begin{theorem}}
\newcommand{\enth}{\end{theorem}}

\newcommand{\noi}{\noindent}

\newcommand{\cC}{{\cal C}}

\newcommand{\cO}{{\cal O}}

\newcommand{\la}{\lambda}
\newcommand{\La}{\Lambda}

\newcommand{\eps}{\varepsilon}

\newcommand{\re}{{\rm e}}
\newcommand{\ri}{{\rm i}}
\newcommand{\rd}{{\rm d}}

\newcommand{\Com}{\mathbb{C}}

\newcommand{\half}{\frac{1}{2}}

\newcommand{\diff}[2]{\frac{\rd #1}{\rd #2}}
\newcommand{\pdiff}[2]{\frac{\partial #1}{\partial #2}}

\newcommand{\tendi}{\rightarrow \infty}
\newcommand{\tendo}{\rightarrow 0}

\def\XXint#1#2#3{{\setbox0=\hbox{$#1{#2#3}{\int}$}
     \vcenter{\hbox{$#2#3$}}\kern-.5\wd0}}

\newcommand{\vertiii}[1]{{\left\vert\kern-0.25ex\left\vert\kern-0.25ex\left\vert #1
    \right\vert\kern-0.25ex\right\vert\kern-0.25ex\right\vert}}

\newcommand{\vertf}[1]{{\left\vert\kern-0.25ex\left\vert\kern-0.25ex\left\vert\kern-0.25ex\left\vert #1
    \right\vert\kern-0.25ex\right\vert\kern-0.25ex\right\vert\kern-0.25ex\right\vert}}

\newcommand{\tfa}{\text{ for all }}

\newcommand{\tas}{\text{ as }}
\newcommand{\tand}{\text{ and }}

\newcommand{\tJ}{\widetilde{J}}

\newcommand{\intFokas}[1]{\int_{#1}^{\infty \re^{\ri \pi/4}}}

\newcommand{\EPhi}{\Phi^{[1]}}
\newcommand{\EPhia}{\Phi}
\usepackage{soul}

\definecolor{escol}{rgb}{0,0,0.8}
\definecolor{estcol}{rgb}{0.8,0,0}
\definecolor{afcol}{rgb}{1,0,0}

\begin{document}

\title{Uniform asymptotics as a stationary point approaches an endpoint}

\author{A.~Fernandez
\thanks{Department of Applied Mathematics and Theoretical Physics, University of Cambridge, Cambridge, CB3 0WA (\texttt{A.Fernandez@damtp.cam.ac.uk})}
,
E.~A.~Spence
\thanks{Department of Mathematical Sciences, University of Bath, Bath, BA2 7AY (\texttt{E.A.Spence@bath.ac.uk})}
,
A.~S.~Fokas
\thanks{Department of Applied Mathematics and Theoretical Physics, University of Cambridge, Cambridge, CB3 0WA, and Viterbi School of Engineering University of Southern California, Los Angeles, CA 90089, USA  (\texttt{T.Fokas@damtp.cam.ac.uk})}
}

\date{\today}

\maketitle

\begin{abstract}
We obtain the rigorous uniform asymptotics of a particular integral where a stationary point is close to an endpoint.
There exists a general method introduced by Bleistein for obtaining uniform asymptotics in this situation.
However, this method does not provide rigorous estimates for the error. Indeed, the method of Bleistein
 starts with a change of variables, which implies that
 the parameter governing how close the stationary point is to the endpoint appears in several  parts of the integrand, and this means that one cannot obtain general error bounds. 
 By adapting the above
  method to our particular integral, we obtain  rigorous uniform leading-order  asymptotics.
We also give a rigorous derivation of the asymptotics to \emph{all} orders of the same integral; the novelty of this second approach is that it does \emph{not} involve a global change of variables. 
\end{abstract}

\section{Introduction}\label{sec:intro}
Stokes and Kelvin established  in the nineteenth century that the main contributions to the large-$t$ asymptotics of the integral
\beqs
\int_\alpha^\beta g(x) \exp\big( \ri th(x)\big)  \rd x,
\eeqs
where the functions $g(x)$ and $h(x)$ are sufficiently smooth,
come from the neighbourhood of the endpoints $\alpha$ and $\beta$, and from the neighbourhood of \emph{stationary points} of $h(x)$, i.e.~points where $h'(x)=0$.

The case when the stationary point is close to an endpoint was considered by Bleistein in \cite{Bl:66}, where he introduced a general algorithm  for
obtaining uniform asymptotics using a global change of variables and integration by parts; 
this work  followed on from analogous treatments of two nearby stationary points by Chester, Friedman, and Ursell \cite{ChFrUr:57}, \cite{Fr:59}, \cite{Ur:65}.
A good description of this general methodology can be found in \cite[Chapter VII]{Wo:89}.
We note, however, that these general algorithms do not give rigorous uniform error estimates, and such estimates must be obtained on a case-by-case basis; we illustrate this statement  below in
our discussion of the particular integral $J_B$, defined by Definition 1.1.

The aim of this paper is to obtain the leading-order asymptotics of the integral $J_B$ with a rigorous uniform error estimate. This estimate is needed for the proof of a variant of the Lindel\"of hypothesis \cite{Fo:17}. 

\begin{definition}[The integral $J_B$]\label{def:tJB}
Let $0<\delta<1$ and $1/2\leq \sigma <1$ be fixed constants, and let $\la$ satisfy
\beq\label{eq:3old}
\frac{t^{\delta-1}}{1-t^{\delta-1}}
\leq  \la \leq t^{1-\delta}-1.
\eeq
Let 
\beq\label{eq:1old}
J_B(t;\la, \delta,\sigma) :=\int_{1-t^{\delta-1}}^{\infty \re^{\ri \phi}} (1-z)^{-1/2}z^{\sigma-1/2} \exp{(\ri t F(z;\la))}\,\rd z,
\eeq
where
\beq\label{eq:2old}
F(z;\lambda) := (1-z) \ln (1-z)+ z \ln z + z \ln \lambda,
\eeq
with the branch cuts of $F$ (as a function of $z$) from $-\infty$ to $0$ and from $1$ to $\infty$.

The angle $\phi$ in the endpoint of integration in \eqref{eq:3old} satisfies
\beq\label{eq:phi1}
0<\phi<\pi/2\quad\text{ when } \ln \la\geq 0,
\eeq
and
\beq\label{eq:phi2}
0<\phi<\arctan \bigg(\frac{\pi}{\big|\ln \la\big|}\bigg)\quad \text{ when } \ln\la<0.
\eeq
\end{definition}

It is straightforward to check that the requirements \eqref{eq:phi1} and \eqref{eq:phi2} are equivalent to demanding that $\Im F(z;\la)>0$ for large $|z|$, i.e.~that the integrand in \eqref{eq:1old} has exponential decay for large $|z|$.

\begin{problem}\label{prob:1}
Find the leading-order
asymptotics of $J_B$ as $t\tendi$ for $\la$ in the range \eqref{eq:3old}, with the error term independent of $\la$.
\end{problem}

Since we are not interested in the dependence of the error term on the parameters $\sigma$ and $\delta$, we will suppress the dependence of $J_B$ (and all other functions) on these two variables; i.e.~we write $J_B=J_B(t;\la)$ only.

\paragraph{Why is Problem \ref{prob:1} difficult?}
Since
\beq\label{eq:1}
\pdiff{F}{z}(z,\la)= \ln \left(\frac{z\la}{1-z}\right),
\eeq
there is a stationary point at
\beq\label{eq:stat_point}
z= \frac{1}{1+\la}.
\eeq
When 
\beqs
\frac{t^{\delta-1}}{1-t^{\delta-1}}< \la \leq t^{1-\delta}-1,
\eeqs
the stationary point is in the interval $(0,1-t^{\delta-1})$, and thus is away from the contour of integration. However, 
when $\la$ equals $\la_c$, defined by 
\beq\label{eq:la_c}
\la_c:=\frac{t^{\delta-1}}{1-t^{\delta-1}}=\frac{1}{t^{1-\delta}-1},
\eeq
the stationary point is at $z=1-t^{\delta-1}$, i.e.~at the endpoint of integration.

The method of Bleistein deals with the situation of a stationary point close to an endpoint by introducing a global change of variables. Before following this method, it is convenient to introduce new variables $\zeta=\zeta(z)$ and $\La=\La(\la)$ so that $\zeta=0$ corresponds to 
$z=1-t^{\delta-1}$ and $\La=0$ corresponds to $\la=\la_c$; i.e.~we let
\beq\label{Lambda-defn}
\la= \la_c(1+\La) \quad\tand \quad z=\frac{(1+ \la_c \zeta)}{1+ \la_c}.
\eeq
In these new variables, the stationary point is at 
\beq\label{eq:stat_point2}
\zeta= \frac{-\La}{1+ \la_c(1+\La)},
\eeq
therefore on the positive real axis, and at the endpoint of integration $\zeta=0$ if $\La=0$.

We then have that 
\beq\label{eq:JBtJB}
J_B(t;\la) = \left(\frac{\la_c}{1+\la_c}\right)^{1/2}\left(\frac{1}{1+\la_c}\right)^{\sigma-1/2} \exp\left( \frac{\ri t f_0(\La,\la_c)}{(1+\la_c)}\right) \, \tJ_B(t;\la),
\eeq
where $f_0$ is defined by 
\beq\label{eq:f_0}
f_0(\La,\la_c):=\ln \left(\frac{\la_c(1+\La)}{1+\la_c}\right)- \la_c \ln\left(\frac{1+\la_c}{\la_c}\right),
\eeq
and $\tJ_B$ is defined by 
\beq\label{eq:0}
\tJ_B(t;\La):= \int_0^{\infty \re^{\ri \phi}} g(\zeta;t) \,\exp\left(\ri t h(\zeta;t,\La)\right) \rd \zeta,
\eeq
where $\phi$ satisfies \eqref{eq:phi1} and \eqref{eq:phi2} as before,
\beq\label{eq:g}
g(\zeta;t):= (1-\zeta)^{-1/2} (1+ \la_c \zeta)^{\sigma-1/2}, \quad\tand \quad 
h(\zeta;t,\La):= \frac{f_1(\zeta;t,\La)}{1+\la_c},
\eeq
with
\beq\label{eq:f1}
f_1(\zeta;t,\La) := \la_c \zeta\Big[ \ln(1+ \Lambda) + \ln(1+ \la_c \zeta) - \ln(1-\zeta) \Big] + \ln (1+ \la_c \zeta) + \la_c \ln(1-\zeta).
\eeq
The branch cut for $\ln(1+ \la_c\zeta)$ is taken from $-1/\la_c$ to $\infty$ on the negative real axis and the branch cut for $\ln(1-\zeta)$ from $1$ to $\infty$ is taken on the positive real axis. Note that the range for $\la$ in \eqref{eq:3old} means that the parameter $\La$ satisfies 
\beq\label{eq:range}
0\leq \Lambda < \frac{t^{1-\delta}-1}{\la_c}-1.
\eeq

The method of Bleistein introduces a global change of variables $u=u(\zeta)$
so that
\beq\label{eq:cov}
h(\zeta;t,\La)- h(0;t,\La)= \half u^2 + au,
\eeq
where $a$ is chosen so that when $\zeta$ is given by \eqref{eq:stat_point2}, $u=-a$. Performing this change of variables in \eqref{eq:0}, and observing that $h(0;t,\La)=0$, we obtain
\beqs
\tJ_B(t;\La)= \intFokas{0} g\left(\zeta(u);t\right) \diff{\zeta}{u}(u) \,\exp \left( \ri t \left( \half u^2 + au\right)\right) \rd u.
\eeqs
The Bleistein method then proceeds by integrating by parts, and gives formal asymptotics of $\tJ_B$. Arguments due to Erdelyi for the case $a=0$ \cite[\S2.9]{Er:56} can be adapted to give a rigorous uniform bound for the error in the general case, i.e.,~for general $g$ and $h$, \emph{under the assumption that the function
\beqs
g\left(\zeta(u);t\right) \diff{\zeta}{u}(u)
\eeqs
is independent of $t$ and $\La$.} Even when $g(\zeta; t)= g(\zeta)$, which is not the case for $\tJ_B$ defined by \eqref{eq:0}, this assumption will not hold in general since the change of variables $\zeta(u)$ (and hence also $(\rd \zeta/\rd u)(u)$) depend on $\La$, via the dependence of $a$ on $\La$.

In this paper, we provide the necessary modifications to these arguments to obtain the rigorous uniform asymptotics of $\tJ_B$ as $t\tendi$.
The solution of Problem \ref{prob:1} is then as follows.

\begin{theorem}[Solution of Problem \ref{prob:1}, i.e.~uniform leading-order asymptotics of $J_B$]\label{thm:1}
Let
\beq\label{eq:omega}
\omega(t,\La):= \sqrt{\frac{\la_c t}{2}} \frac{\ln(1+\Lambda)}{1+\la_c}.
\eeq
Observe that $\omega\geq 0$ since $\La\geq 0$.
The leading-order asymptotics of $J_B$ are given by \eqref{eq:JBtJB}, with the leading-order asymptotics of $\tJ_B$ given by
\beq\label{eq:result1}
\tJ_B(t;\la)= 
\re^{-\ri \omega^2}
\sqrt{\frac{2}{\la_c t}} \left( \int_\omega^{\infty \re^{\ri \pi/4}}\re^{\ri \xi^2}\rd \xi\right) \Big(1+ o(1)\Big),
\eeq
where the $o(1)$ is independent of $\La$.

If $\Lambda$ is such that $\omega=\cO(1)$ as $t\tendi$ (e.g.~$\La=0$), then the integral on the right-hand side of \eqref{eq:result1} is an $\cO(1)$ quantity independent of $\La$.
Furthermore, if $\Lambda$ is such that $\omega\rightarrow\infty$ as $t\tendi$, then
\beq\label{eq:result2}
\tJ_B(t;\la)= 
\sqrt{\frac{2}{\la_c t}}
\left( \frac{-1}{2\ri \omega}+ \cO\left( \frac{1}{\omega^3}\right)\right) \Big(1+o(1)\Big),
\eeq
as $t\tendi$, where both the $o(1)$ and the omitted constant in the $\cO( 1/\omega^3)$
are independent of $\Lambda$.
\end{theorem}

The integral on the right-hand side of \eqref{eq:result1} is a Fresnel-type integral, and can be expressed in terms of the special function $\mathscr{F}(z)$ defined in \cite[Equation 7.2.6]{Di:17}.

\bre
\emph{\textbf{(Transition between ``stationary-point" and ``integration-by-parts" contributions)}}
In the asymptotics \eqref{eq:result1}, 
$\la_c t$ plays the role of the large parameter (recall from \eqref{eq:la_c} that $\la_c \sim t^{\delta-1}$ as $t\tendi$, so $\la_c t \sim t^\delta \tendi$ as $t\tendi$).

When $\La$ is such that $\omega=\cO(1)$ as $t\tendi$, the right-hand side of \eqref{eq:result1} is $\cO((\la_c t)^{-1/2})$,
 i.e., the asymptotics expected from a stationary point.
 When $\La$ is such that $\omega\tendi$ as $t\tendi$, \eqref{eq:result2} implies that the right-hand side of \eqref{eq:result1} is $\cO((\la_c t)^{-1})$, i.e., the asymptotics expected from integration by parts away from a stationary point.
\ere

\paragraph{Obtaining the solution to Problem \ref{prob:1} \emph{without} a global change of variables.}

When one obtains the formal asymptotics of standard stationary-phase-type integrals (i.e.~those without the stationary point approaching an endpoint), one splits the integral, uses local expansions near the stationary point, and then uses integration by parts away from the stationary point (see, e.g., \cite[\S6.5]{BeOr:78}).
Similarly, for the formal asymptotics of Laplace-type integrals, one uses local expansions near the points at which the exponent is maximised, and integration by parts away from these points (see, e.g., \cite[\S6.4]{BeOr:78}).

For the rigorous justification of these asymptotics, however, the standard approach is to first make a global change of variables (in the same spirit as \eqref{eq:cov} above); see, e.g., 
\cite[\S2.4, \S2.9]{Er:56}, \cite[Chapter 2 \S1 \S3]{Wo:89}, and \cite[\S3.3, \S5.3]{Mi:06}.

It is rare to see examples in the literature where rigorous asymptotics are obtained 
\emph{without} first making a global change of variables, but via directly splitting the integral and using local expansions and integration by parts; one notable exception is \cite[\S6.4, Examples 7 and 8]{BeOr:78}.

It is therefore a challenging question whether the rigorous \emph{uniform} leading-order asymptotics of $J_B$ can be obtained \emph{without} first making a global change of variables. In this paper we show that this is indeed possible; in fact, we go even further by obtaining the asymptotics to \emph{all} orders of $J_B$, in the most important case when $\sigma=1/2$.

\begin{theorem}[Asymptotics of $J_B$ to all orders when $\sigma=1/2$]
\label{JB1+2-theorem}
In the case $\sigma=1/2$, the asymptotics of $J_B$ to all orders are given by
\begin{align}\nonumber
J_{B}(t;\la)=&\sum_{j=1}^{m-3}T_j(t;\la)+\cO\left((4m-5)!! t^{-\frac{1}{2}-\frac{(4m-5)\delta}{2}}(\ln t)^{(4m-3)/2}a^{-4m+4}\right) \\
&+\exp\Big(\ri tF(1-t^{\delta-1};\lambda)-
\ri\omega^2\Big)\,t^{-1/2}\sqrt{\frac{2}{1+\lambda_c}}\left(\int_{\omega}^{\omega + a \sqrt{\la_c t/2}}\re^{\ri\xi^2}\,\mathrm{d}\xi\right)
+\cO\left(t^{-\frac{1}{2}+\frac{3\delta}{2}}a^4\right),\label{JB1+2-asymptotics}
\end{align}
for any natural number
 $m\geq 4$, where $a\gg t^{-\delta/2}$ is such that 
\beq\label{a-assumption-sandwich}
t^{\epsilon-\frac{\delta}{2}+\frac{\delta}{4m+2}}\ll a\ll t^{-\frac{\delta}{2}+\frac{\delta}{4m-2}}
\eeq
for some $\epsilon>0$, and the summands $T_j$ are defined exactly by \eqref{eq:Tj} below and can be estimated by (\ref{Tj-estimate-a}).
\end{theorem}

The result of Theorem \ref{JB1+2-theorem} is a uniform expansion in the sense that the $\cO(\cdot)$ terms are independent of $\la$, but it is not quite a uniform asymptotic expansion in the sense of, e.g., \cite[Chapter VII, \S1]{Wo:89} since the ordering of the $T_j$ terms and the term involving the integral is not immediately specified.

\paragraph{Outline of the paper.}

In \S\ref{sec:outline} we recap Bleistein's ``global change of variable + integration by parts" method, supplemented with ideas from Erdelyi \cite[\S2.9]{Er:56} to bound the error term. In \S\ref{sec:proof} we prove Theorem \ref{thm:1} by adapting the method in \S\ref{sec:outline}.
In \S\ref{sec:Arran} we prove Theorem \ref{JB1+2-theorem}, and also show how the result of Theorem \ref{thm:1} follows from that of Theorem \ref{JB1+2-theorem}.

\section{Recap of Bleistein's ``global change of variable + integration by parts" method}\label{sec:outline}

In this section we give an overview of the  ``global change of variable + integration by parts" method for obtaining uniform asymptotics for a stationary point near an endpoint.
As discussed in \S\ref{sec:intro}, this method was introduced by Bleistein \cite[Section 5]{Bl:66}
and appears in, e.g., \cite[Chapter VII]{Wo:89}.

We follow the approach of Bleistein, but perform the integration by parts slightly differently, following Erdelyi's treatment of the method of stationary phase in \cite[\S2.9]{Er:56}. The latter treatment makes it
easier to rigorously estimate the error in the leading-order asymptotics for $\tJ_B$. These different approaches are
 discussed further  in \S\ref{sec:Erd} and \S\ref{sec:BW}.

\subsection{Notation and assumptions}

Let
\beq\label{eq:E1}
J(t;\La):= \int_0^{\infty \re^{\ri \phi}} g(\zeta) \,\exp\left( \ri t h(\zeta;\La)\right) \rd \zeta,
\eeq
with $\La\in [0, \La^*]$.

We assume that $g$ and $h$ are analytic functions of $\zeta$, apart from possible branch points and branch cuts lying away from the contour of integration.
We assume that $h$ has one stationary point in $\Com$, whose location depends on $\La$. We let the location of this stationary point be at $\zeta= s(\La)$.
Without loss of generality we assume that $s(0)=0$, i.e.,~when $\La=0$ the stationary point is at the end point of integration.
We assume that $\Re s(\La)\leq 0$ (so that the stationary point is not on the contour of integration for $\La>0$). The case when $s(0)=0$ but $\Re s(\La)>0$ for $\La>0$ (i.e.~the stationary point is on the contour of integration for $\La>0$, at least after a suitable contour deformation) can be treated similarly -- see \cite[Chapter VII, \S3]{Wo:89}.

We assume that $\phi$ is chosen so that the integral converges. In particular we assume that, when $\zeta= |\zeta|\exp(\ri \phi)$, $\Im h(\zeta;\La)\tendi$ as $|\zeta|\tendi$. We also assume that $g(\zeta)$ grows at most polynomially in $\zeta$ as $|\zeta|\tendi$.

The goal is to find the asymptotics of $J(t;\La)$ as $t\tendi$, valid for $\La\in [0,\La^*]$.

\bre[The form of the exponent]
The method we outline in this section would also apply to the integral
\beqs
\tJ(t;\La):= 
\int_0^{\infty \re^{\ri \phi}} g(\zeta)\, \exp\left(- t\,h(\zeta;\La)\right)  \rd \zeta,
\eeqs
with $\Re h(\zeta;\La)\tendi$ as $|\zeta|\tendi$ (and is presented for this case in \cite[Section 6]{Bl:66} and \cite[Chapter VII, \S3]{Wo:89}); i.e.~our choice in \eqref{eq:E1} of $J(t;\la)$ being a ``stationary-phase-type" integral is not restrictive.
\ere

\subsection{Definition of the global change of variables}\label{sec:gcov}

The simplest example of a function satisfying the assumptions on $h$ above is
\beqs
h(\zeta;a) = \half \zeta^2 + a\zeta.
\eeqs
We therefore seek a change of variables $\zeta= \zeta(u)$ such that
\beq\label{eq:E2}
h(\zeta;\La) - h(0;\La) = \half u^2 + au;
\eeq
observe that the endpoint $\zeta=0$ is mapped to $u=0$.
We then fix the value of $a$ by demanding that $\zeta= s(\La)$ (i.e., $\zeta$ is at the stationary point) when $u=-a$. We choose this sign-convention for $a$ motivated by the case when (i) $s(\La)$ is real and (ii) $h(\zeta;\La)$ is real for real $\zeta$, since then $a\geq 0$ (and so the stationary point at $u=-a$ is on the negative real axis).

Thus, we let
\beq\label{eq:E3}
a:=
\sqrt{2}\sqrt{\big.h(0;\La)- h(s(\La);\La)},
\eeq
and then
\beq\label{eq:E4}
u= - \sqrt{2}\sqrt{\big.h(0;\La)-h(s(\La);\La)} + \sqrt{2}\sqrt{\big.h(\zeta;\La) - h(s(\La);\La)},
\eeq
with the branch cut chosen on the negative real axis.
The change of variables
\eqref{eq:E4} is well-defined since, from \eqref{eq:E2},
\beqs
\diff{u}{\zeta} = \diff{h}{\zeta}(\zeta;\La)\frac{1}{ \sqrt{2}\sqrt{\big.h(\zeta;\La)- h(s(\La);\La)}},
\eeqs
and both numerator and denominator have a simple zero at $\zeta= s(\La)$.

Under this change of variables, $J$ defined by \eqref{eq:E1} becomes
\beq\label{eq:E6}
J(t;\La)= \exp\left( \ri t h(0;\La)\right)\intFokas{0} g\left(\zeta(u)\right) \diff{\zeta}{u}(u) \exp \left( \ri t \left( \half u^2 + au\right)\right)  \rd u,
\eeq
where we have fixed the endpoint of integration to give the fastest decay of the exponent as $|u|\tendi$.
We now assume, without loss of generality, that $h(0;\La)=0$, so that the factor in front of the integral on the right-hand side of \eqref{eq:E6} is one.

\subsection{Integration by parts (via Erdelyi)}\label{sec:Erd}

There are now two options:
\ben
\item Integrate \eqref{eq:E6} directly by parts (following Erdelyi's treatment in \cite[\S2.9]{Er:56} of \eqref{eq:E6} with $a=0$).
\item Let
\beq\label{eq:E4c}
F(u):= g\left(\zeta(u)\right) \diff{\zeta}{u}(u),
\eeq
define $a_0, b_0$, and $G(u)$ so that
\beq\label{eq:BW1}
F(u) = a_0 + b_0 u + u(u+a) G(u),
\eeq
and then integrate by parts (following Bleistein \cite[Section 5]{Bl:66}).
\een
In the case of $\tJ_B$, we are only interested in obtaining the leading-order term with a rigorous error estimate, and
it turns out that this goal is more easily achieved via Option 1 rather than Option 2.
If one wants to obtain a uniform asymptotic expansion to all orders, it turns out that Option 2 is the better option. 

We therefore concentrate on Option 1, but then outline Option 2 in \S\ref{sec:BW}.
The key point about Option 2 is that the introduction of $u+a= \diff{}{u}( \half u^2 +au)$ in the integrand means that the sequence of (special) functions, with respect to which the asymptotics are found, can easily be seen to be uniform with respect to $a$, and hence with respect to $\La$ (see \eqref{eq:BW3} and the associated discussion below).

\ble[Integration by parts using Cauchy's formula for repeated integration]\label{lem:Cauchy}
Let $\widetilde{F}(r):= F(r \re^{\ri \pi/4})$ and $\widetilde{f}(r):= f(r \re^{\ri \pi/4})$. If $\widetilde{F}$, $\widetilde{f} \in C^\infty[0,\infty)\cap L^1(0,\infty)$ then, for $n\in \mathbb{Z}$,
\beq\label{eq:E4a}
\intFokas{0} F(u) f(u)\, \rd u = \sum_{k=0}^n  F^{(k)}(0) \, \phi_f^{[k+1]}(0) +
 \intFokas{0} F^{(n+1)}(u) \, \phi_f^{[n+1]}(u) \, \rd u,
\eeq
where
\beqs
\phi_f^{[k+1]}(u):= \frac{1}{k!} \intFokas{u} (v-u)^k f(v)\, \rd v.
\eeqs
\ele

\bpf[Outline of the proof]
This formula follows from integration by parts,
using the fact that
\beqs
\intFokas{u} \phi_f^{[k]}(v) \,\rd v = \phi_f^{[k+1]}(u),
\eeqs
and thus $\phi_f^{[k+1]}(u)$ is the $(k+1)$th repeated integral of $f$ -- this is Cauchy's formula for repeated integration.
\epf

\

We apply Lemma \ref{lem:Cauchy} to our integral \eqref{eq:E6}
with $F(u)$ defined by \eqref{eq:E4c} and
\beq\label{eq:E4ca}
f(u):=\exp\left( \ri t \left(\half u^2 + au\right)\right).
\eeq
It is convenient to use the notation $\Phi^{[k+1]}$ for $\phi_f^{[k+1]}$ with $f$ defined by \eqref{eq:E4ca}, and explicitly indicate that $\Phi^{[k+1]}$ depends on $t$ and $\La$ as well as $u$, i.e.,~we define $\Phi^{[k+1]}(u;t,\La)$ by
\beq\label{eq:E8a}
\Phi^{[k+1]}\left(u;t,\La\right):= \frac{1}{k!} \intFokas{u} (v-u)^k \exp\left( \ri t \left(\half v^2 + av\right)\right)  \rd v.
\eeq
Then \eqref{eq:E4a} becomes
\beq\label{eq:E4d}
J(t;\La)=\sum_{k=0}^n F^{(k)}(0) \, \Phi^{[k+1]}\left(0;t,\La\right) +
 \intFokas{0} F^{(n+1)}(u) \, \Phi^{[n+1]}\left(u;t,\La\right)  \rd u.
\eeq

\bre[In general $F$ depends on $\La$]\label{rem:Fdepend}
In the rest of this section we assume that $F(u)$ in \eqref{eq:E4c} is  independent of $\La$ and $t$ (as is done in \cite{Bl:66} and \cite{Wo:89}).
As in \S\ref{sec:intro} we emphasise, however, that this is not in general true. Indeed, the change of variables \eqref{eq:E4} depends on $\La$ and thus so does $F(u)$ (defined in \eqref{eq:E4c}).
\ere

\subsection{The leading order behaviour: transition between ``stationary-point" and ``integration-by-parts" contributions}\label{sec:phi}

Before proving that \eqref{eq:E4d} is indeed an asymptotic expansion of $J$ as $t\tendi$, we look at the leading-order term $F(0) \EPhi(0;t,\La)$. Furthermore, since we are assuming that $F(0)$ is independent of $t$ and $\La$, we focus on $\EPhi(0;t,\La)$. We hope to be able to see the transition between the $t^{-1/2}$ ``stationary-point" contribution and the $t^{-1}$ ``integration-by-parts" contribution.

Using the change of variables $\xi =  (v+a)\sqrt{t/2}$ in \eqref{eq:E8a} we have
\beqs
\EPhi\left(u;t,\La\right)= \sqrt{\frac{2}{t}} \exp\left( - \frac{\ri a^2}{2}t\right) \intFokas{(u+a)\sqrt{t/2}} \exp(\ri \xi^2) \, \rd \xi.
\eeqs
The above integral can be expressed in terms of the Fresnel integral $\mathscr{F}(z)$ \cite[Equation 7.2.6]{Di:17}, but we will not need this connection for what follows.

The leading-order behaviour of $J$ is therefore dictated by 
\beq\label{eq:E11}
\EPhi\left(0;t,\La\right)= \sqrt{\frac{2}{t}} \exp\left( - \frac{\ri a^2}{2}t\right) \intFokas{a\sqrt{t/2}} \exp(\ri \xi^2) \, \rd \xi.
\eeq
We now determine the $t\tendi$ asymptotics of \eqref{eq:E11} for different values of $a$,
 where $a$ depends on $\La$ via \eqref{eq:E3}. Our aim is to have   asymptotics explicit in $\La$,
 thus we make everything explicit in $a$.

If $|a|\sqrt{t}= \cO(1)$ as $t\tendi$, then the integral in \eqref{eq:E11} is $\cO(1)$ and $\EPhi(0;t,\La) = \cO(t^{-1/2})$; i.e., the asymptotics expected from a stationary point.
If $|a|\sqrt{t}\tendi$ as $t\tendi$ then,
by integration by parts (or by recalling the asymptotics of Fresnel integrals \cite[\S7.12]{Di:17}), we find that
\beq\label{eq:9alt}
\intFokas{z} \re^{\ri \xi^2} \, \rd \xi = \re^{\ri z^2} \left( -\frac{1}{2\ri z} + \cO\left( \frac{1}{z^3}\right)\right),
\eeq
and so
\beq\label{eq:E13}
\EPhi\left(0;t,\La\right) = \sqrt{\frac{2}{t}}\left( -\frac{1}{\ri a \sqrt{2t}} + \cO \left( \frac{1}{(a\sqrt{t})^3}
\right)\right).
\eeq
If $|a|= \cO(1)$, then $\EPhi(0;t,\La) = \cO(t^{-1})$; i.e., the asymptotics expected from integration by parts away from a stationary point.

\subsection{Rigorously estimating the error term (via Erdelyi)}\label{sec:rigor}

We now outline how it can be shown that \eqref{eq:E4d} is indeed an asymptotic expansion of $J$; i.e.,
\beq\label{eq:E7}
\intFokas{0} F^{(n+1)}(u) \,\Phi^{[n+1]}\left(u;t,\La\right) \rd u =o\Big( F^n(0)\,\Phi^{[n+1]}\left(0;t,\La\right)\Big) \quad\tas t\tendi.
\eeq

Since we are only interested in the leading-order asymptotics of our particular example $\tJ_B$, we will present the method for proving \eqref{eq:E7} for $n=0$; the case of general $n$ is very similar (see \cite[\S2.9]{Er:56}).

Furthermore, since we have assumed that $F(0)$ is independent of $\La$ and $t$ (see Remark \ref{rem:Fdepend}), we only need to prove that
\beq\label{eq:E8}
\intFokas{0} F'(u) \EPhi\left(u;t,\La\right) \rd u = o\Big( \EPhi\left(0;t,\La\right)\Big) \quad\tas t\tendi;
\eeq
that is, given $\eps >0$, there exists a $C>0$ (independent of $\La$ and $t$) such that
\beq\label{eq:E8aa}
\left|\intFokas{0} F'(u) \EPhi\left(u;t,\La\right) \rd u \right|\leq C \eps\left| \EPhi\left(0;t,\La\right)\right|,
\eeq
for $t$ sufficiently large.
In what follows, we will drop the superscript $[1]$ on $\Phi$; i.e.,~we define $\Phi(u;t,\La):= \EPhi(u;t,\La)$.

The main idea for proving \eqref{eq:E8aa} is that, when $u= |u|\re^{\ri \pi/4}$ with $|u|\geq \delta$ (with $\delta$ independent of $t$),
$\EPhia(u;t,\La)$ decays faster than $\EPhia(0;t,\La)$ as $t\tendi$. Indeed, for such a $u$, integration by parts via \eqref{eq:9alt} shows that the leading-order behaviour of $\EPhia(u;t,\La)$ contains the term
\beqs
\exp \left( \frac{\ri t}{2} \left(u^2 + 2 ua\right)\right);
\eeqs
since $\arg u = \pi/4$ this leads to exponential decay in $t$, as opposed to the algebraic decay in $\EPhia(0;t,\La)$ seen in \eqref{eq:E11} and \eqref{eq:E13}.
For $n \geq 1$, this argument is slightly different -- see \cite[\S2.9, Equations 14 and 16]{Er:56} -- but the faster  decay of $\EPhia(u;t,\La)$ in comparison to $\EPhia(0;t,\La)$ still holds.

Based on this observation, the method consists of the following steps.
\ben
\item Control $\EPhia(u;t,\La)$ in terms of $\EPhia(0;t,\La)$ for $u=|u|\re^{\ri \pi/4}$ (for $t$ sufficiently large) by proving, for example, that
\beq\label{eq:E15}
\big|\EPhia(|u|\re^{\ri \pi/4};t,\La)\big| \leq C_1 \big|\EPhia(0;t,\La)\big|,
\eeq
for $|u|>0$ and with $C_1$ independent of $t$ and $\La$.
\item Given $\eps>0$, choose $\delta_1(\eps)$ be such that
\beq\label{eq:E16}
\int_0^{\delta_1(\eps)\re^{\ri \pi/4}} \big|F'(u)\big| \, \rd u \leq \eps.
\eeq
Note that, under our assumption that $F$ is independent of $\La$ and $t$, $\delta_1(\eps)$ is independent of $\La$ and $t$.
\item Use the exponential decay of $\EPhia(u;t,\La)$ for $u=|u|\re^{\ri \pi/4}$ with $|u|\geq \delta_1(\eps)$ to show that
\beq\label{eq:E17}
\left|
\int_{\delta_1(\eps)}^{\infty \re^{\ri \pi/4}} F'(u) \EPhia\left(u;t,\La\right) \rd u
\right| \leq C_2 \,\eps \,\big|\EPhia\left(0;t,\La\right)\big|,
\eeq
with $C_2$ independent of $t$ and $\La$, and for $t$ sufficiently large.
\item Combine the results \eqref{eq:E15}, \eqref{eq:E16}, and \eqref{eq:E17}, via
\begin{align*}
\left|\intFokas{0}F'(u)\EPhia\left(u;t,\La\right)  \rd u \right| &\leq
\left|\int_0^{\delta_1(\eps)\re^{\ri \pi /4}}F'(u)\EPhia\left(u;t,\La\right)  \rd u \right|
+
\left|\int_{\delta_1(\eps)}^{\infty \re^{\ri \pi/4}}F'(u)\EPhia\left(u;t,\La\right) \rd u \right|,\\
& \leq \eps \left( C_1+ C_2\right) \big|\EPhia\left(0 ;t,\La\right)\big|,
\end{align*}
which proves the required result \eqref{eq:E8aa}.
\een

\subsection{Integration by parts (via Bleistein)}\label{sec:BW}

We now outline the method of Bleistein. As stated at the beginning of \S\ref{sec:Erd}, this method starts by rewriting the integrand before performing the integration by parts.
Indeed, we define $a_0$, $b_0$, and $G(u)$ by
\beqs
a_0:= F(0), \qquad b_0 := \frac{1}{a}\left( F(0)- F(-a)\right), \quad\tand\quad
G(u):= \frac{F(u)-a_0 -b_0u}{u(u+a)},
\eeqs
so that \eqref{eq:BW1} holds.
Substituting  \eqref{eq:BW1} into the expression \eqref{eq:E6} for $J$ (and recalling that we have assumed that $h(0,\La)=0$), we find
\beqs
J(t;\La)= \frac{a_0}{\sqrt{t}} U(a\sqrt{t}) - \ri \frac{b_0}{t} U'(a\sqrt{t}) + \intFokas{u} u(u+a) G(u) \exp\left( \ri t \left(\half u^2 +au\right)\right) \rd u,
\eeqs
where
\beqs
U(x):= \intFokas{0} \exp\left( \ri \left(\half v^2 + xv\right)\right)\rd v;
\eeqs
(note that $U(x)$ can be expressed in terms of the parabolic cylinder function; see \cite[Chapter VII, Equation 3.26]{Wo:89}).
Then, integrating by parts using the fact that
\beqs
\diff{}{u} \left[\frac{1}{\ri t}\exp\left( \ri t \left(\half u^2 +au\right)\right) \right]= (u+a) \exp\left( \ri t \left(\half u^2 +au\right)\right),
\eeqs
we find
\beq\label{eq:E18}
J(t;\La)= \frac{a_0}{\sqrt{t}} U(a\sqrt{t}) - \ri \frac{b_0}{t} U'(a\sqrt{t}) - \frac{1}{\ri t}\intFokas{u} \diff{}{u}\big[ u\, G(u)\big] \exp\left( \ri t \left(\half u^2 +au\right)\right) \rd u.
\eeq
The key point now is that the integral on the right-hand side of \eqref{eq:E18} is of the same form as $J$.
Repeating this process, one  obtains an expansion in terms of $U(a\sqrt{t})$ and $U'(a\sqrt{t})$ or, more precisely, with respect to the asymptotic sequence
\beq\label{eq:BW3}
\frac{1}{t^{n+1/2}} U(a\sqrt{t}) + \frac{1}{t^{n+1}} U'(a\sqrt{t}),
\eeq
which is uniform with respect to the parameter $a$, and hence with respect to $\La$.
This is in contrast to Erdelyi's approach above, where the asymptotic sequence is $\Phi^{[k+1]}(u;t,\La)$ defined by \eqref{eq:E8a}. Since the integrand of each $\Phi^{[k+1]}(u;t,\La)$ is different for each $k$, more work is needed to obtain a sequence that can be proven to be uniform with respect to $\La$.

From our point of view of estimating the error in the leading-order term, it is easier  to estimate the left-hand side of \eqref{eq:E8} involving $F$ than to estimate the integral in \eqref{eq:E18} involving $G$, hence we focus on Erdelyi's method as opposed to Bleistein's.

\section{Proof of Theorem \ref{thm:1} via the method outlined in \S\ref{sec:outline}}\label{sec:proof}

As outlined in \S\ref{sec:outline} there are 4 steps:
\ben
\item[Step 1:] Make a global change of variables (\S\ref{sec:gcov}).
\item[Step 2:] Integrate by parts (\S\ref{sec:Erd}).
\item[Step 3:] Find the asymptotics of $\Phi(u;t,\La)$ (\S\ref{sec:phi}).
\item[Step 4:] Rigorously estimate the error term by splitting the integral (\S\ref{sec:rigor}).
\een

\subsection{Step 1: Global change of variables}

We make a global change of variables such that
\beqs
f_1(\zeta;t,\La) - f_1(0;t,\La) = \alpha u^2 + \beta u.
\eeqs
In \S\ref{sec:gcov} we chose $\alpha=1/2$ and then chose $\beta$ to fix the location of the stationary point. In the case of $\tJ_B$, it turns out to be slightly more convenient to chose $\alpha$ and $\beta$  so that $\zeta =u + o(1)$ as $u\tendo$.

\ble[The integral under the global change of variables]\label{lem:1}
Define the variable $u$ implicitly by
\beq\label{eq:2}
f_1(\zeta;t,\La) = \frac{\la_c}{2}(1+ \la_c)u^2 + \la_c \ln(1+ \La) u.
\eeq
Then
\beq\label{eq:3}
\diff{\zeta}{u}= \frac{ \ln(1+\La) + (1+\la_c) u}{\ln(1+ \La) + \ln(1+ \la_c \zeta)- \ln(1-\zeta)},
\eeq
and
\beq\label{eq:2a}
\tJ_B= \int_0^{\infty \re^{\ri \pi/4}} g\left(\zeta(u);t\right) \diff{\zeta}{u}(u) \exp \left( \frac{\ri \la_c t}{2} \left(u^2 + \frac{2\ln(1+\La)  }{1+ \la_c}u\right)\right)\rd u.
\eeq
\ele

\bpf[Proof of Lemma \ref{lem:1}]
Differentiating the definition of $f_1$ \eqref{eq:f1}, we have
\beq\label{eq:1}
\diff{f_1}{\zeta}(\zeta) = \la_c \ln\left( \frac{(1+\Lambda)(1+ \la_c \zeta)}{1-\zeta}\right),
\eeq
and differentiating \eqref{eq:2} we find
\beq\label{eq:s1}
\diff{f_1}{\zeta} = \la_c (1+\la_c) u \diff{u}{\zeta} + \la_c \ln(1+ \La)\diff{u}{\zeta};
\eeq
combining these two equations, we obtain \eqref{eq:3}.

From \eqref{eq:1} and \eqref{eq:stat_point2}, we see that, when $\La>0$, $\rd f_1/\rd \zeta\neq 0$ for all $\zeta$ on the contour of $\tJ_B$ \eqref{eq:0}. By \eqref{eq:s1}, $\rd \zeta /\rd u\neq 0$ for all $u$ on the image of the contour of $\tJ_B$, and thus the change of variables \eqref{eq:2} is well-defined for all $\zeta$ on the contour of $\tJ_B$.

When $\La= 0$, $(\rd f_1/\rd \zeta)(0)=0$, but l'H\^{o}pital's rule applied to \eqref{eq:3} implies that $(\rd \zeta/\rd u)(0)=1$, and the change of variables \eqref{eq:2} is again well-defined
for all $\zeta$ on the contour of $\tJ_B$.

Making the change of variables \eqref{eq:2} in the integral \eqref{eq:0} we obtain \eqref{eq:2a}. Recall that in \eqref{eq:0} the endpoint of integration was fixed by the requirement that $\Im f_1(\zeta;t,\La)>0$ for large $|\zeta|$. From \eqref{eq:2} we see that this requirement becomes $\Im u^2>0$ for large $|u|$.
\epf

\bre[Taylor series expansion of $f_1(\zeta;t,\La)$ about $\zeta=0$]\label{rem:Taylor}
By Taylor's theorem,
\beqs
f_1(\zeta;t,\La) = \la_c \ln(1+\La) \zeta + \frac{\la_c}{2} (1+\la_c) \zeta^2 + \cO (\la_c \zeta^3) \quad\tas \zeta\tendo,
\eeqs
and so the change of variables \eqref{eq:2} means that $u= \zeta + \cO(\zeta^2)$ as $\zeta\tendo$;  one can check that the omitted constant in the $\cO(\zeta^2)$ can be taken to be independent of $\La$.
\ere

\subsection{Step 2: Integration by parts}

\ble\label{lem:ibps}
Let
\beq\label{eq:F}
F(u;t,\La) := g\left(\zeta(u);t\right) \diff{\zeta}{u}(u),
\eeq
(where $g(\zeta;t)$ is defined by \eqref{eq:g})
and
\beq\label{eq:5}
\Phi(u;t,\La):= \int_{u}^{\infty \re^{\ri \pi/4}} \exp \left( \frac{\ri \la_c t}{2} \left(v^2 + \frac{2\ln (1+\La)}{1+\la_c}v\right)\right) \rd v.
\eeq
Then
\beq\label{eq:key}
\tJ_B(t;\La)= \Phi(0;t,\La) + \intFokas{0} F'(u;t,\La) \, \Phi(u;t,\La) \, \rd u,
\eeq
where $'$ denotes differentiation with respect to $u$.
\ele

\bpf[Proof of Lemma \ref{lem:ibps}]
This follows from \eqref{eq:2a} and \eqref{eq:E4a} with $F$ defined by \eqref{eq:F}, and $f$ defined by
\beqs
f(u;t,\La) := \exp\left( \frac{\ri \la_c t}{2} \left(u^2 + \frac{2\ln (1+\La)}{1+\la_c}u\right)\right).
\eeqs
Now, when $u=0$, $\zeta=0$ and thus
$g(\zeta(0);t)= g(0;t)=1.$
The expression \eqref{eq:3} implies that  $(\rd \zeta/\rd u)(0)=1$ when $\La>0$, and in the Proof of Lemma \ref{lem:1} it was shown that $(\rd \zeta/\rd u)(0)=1$ when $\La=0$; therefore $F(0)=1$.
\epf

\subsection{Step 3: Asymptotics of $\Phi(u;t,\La)$}

We now need to compute the asymptotics of $\Phi(u;t,\La)$ as $t\tendi$, uniformly for $\La$ in the range \eqref{eq:range} in the cases $u=0$ and $u= |u|\re^{\ri \pi/4}$ with $|u|>0$.

\ble[Asymptotics of $\Phi(u;t,\La)$]\label{lem:Phi}

\

\noi (i) $u=0$: with $\omega$  defined by \eqref{eq:omega}, we have
\beq\label{eq:7a}
\Phi\left(0;t,\La\right)= \exp\left( - \frac{\ri \la_c t}{2} \left( \frac{\ln (1+ \La)}{1+\la_c} \right)^2\right) \sqrt{\frac{2}{\la_c t}} \left( \intFokas{\omega} \re^{\ri \xi^2}\, \rd \xi\right).
\eeq
If $\Lambda$ is such that $\omega= \cO(1)$ 
as $t\tendi$ (e.g.~$\La=0$), then the integral on the right-hand side of \eqref{eq:result1} is an $\cO(1)$ quantity as $t\tendi$ (with the omitted constant independent of $\La$).
Furthermore, if $\Lambda$ is such that $\omega\rightarrow\infty$ as $t\tendi$, then
\beq\label{eq:7b}
\Phi(0;t,\La) = 
\sqrt{\frac{2}{\la_c t}}
\left( \frac{-1}{2\ri \omega} + \cO\left( \frac{1}{\omega^3}\right)\right) \quad \tas t\rightarrow\infty,
\eeq
where the omitted constant in $\cO(1/\omega^3)$ is independent of $\La$.

\

\noi (ii) $u= |u|\re^{\ri \pi/4}$ with $|u|>0$:
\begin{align}\nonumber
\Phi(u;t,\La) = &
\sqrt{\frac{2}{\la_c t}} \exp\left(  \frac{\ri\la_c t}{2}\left(u^2 + \frac{2 \ln(1+\La)}{1+ \la_c}\right)u\right)
\\
&
\qquad\qquad \times
\left[ -\left(2 \ri \left( \sqrt{\frac{\la_c t}{2}}u + \omega\right)\right)^{-1} + \cO\left( \left( \sqrt{\frac{\la_c t}{2}}u + \omega\right)^{-3}\right)\right]\tas\sqrt{\la_c t} |u|\tendi,
\label{eq:8a}
\end{align}
independently of whether $\omega = \cO(1)$ or $\omega\tendi$, where the omitted constant in the $\cO(\cdot)$ is independent of $\La$.
\ele

\bpf[Proof of Lemma \ref{lem:Phi}]
Making the change of variable
\beqs
\xi = \sqrt{\frac{\la_c t}{2}} \left( v + \frac{\ln (1+ \La)}{1+ \la_c}\right)
\eeqs
in \eqref{eq:5}, we have
\beq\label{eq:8b}
\Phi(u;t,\La) = \exp\left( - \frac{\ri \la_c t}{2} \left( \frac{\ln (1+ \La)}{1+ \la_c}\right)^2\right) \sqrt{\frac{2}{\la_c t}} \intFokas{\sqrt{\frac{\la_c t}{2}} u+ \omega} \re^{\ri \xi^2} \rd \xi,
\eeq
where $\omega$ is defined by \eqref{eq:omega}. The expression \eqref{eq:7a} follows immediately, and the asymptotics \eqref{eq:7b} and \eqref{eq:8a} then follow  from \eqref{eq:9alt}.
\epf

\

When estimating the integral in \eqref{eq:key} in Step 4, we also need the following lemma (see \eqref{eq:E15}).

\ble\label{lem:3.4}
There exists a $C_1>0$, independent of $t$ and $\La$, such that
\beqs
\Big| \Phi\left( |u|\re^{\ri \pi/4};t,\La\right)\Big| \leq C_1 \big|\Phi(0;t,\La)\big|,
\eeqs
for all $|u|>0$, for all $t>0$, and for all $\La$ in the range \eqref{eq:range}.
\ele

\bpf[Proof of Lemma \ref{lem:3.4}]
From \eqref{eq:8b},
\beqs
 \Phi\left( |u|\re^{\ri \pi/4};t,\La\right) = \exp \left( - \frac{\ri \la_c t}{2} \left( \frac{\ln(1+ \La)}{1+ \la_c}\right)^2 \right) \sqrt{\frac{2}{\la_c t}} \Psi\left( \sqrt{\frac{\la_c t}{2}} |u|; \omega\right),
\eeqs
where
\beqs
\Psi(a;\omega) := \intFokas{a\re^{\ri \pi/4} + \omega} \re^{\ri \xi^2}\, \rd \xi.
\eeqs
It is therefore sufficient to prove that there exists a constant  $C_2>0$ such that
\beqs
|\Psi(a;\omega) - \Psi(0, \omega)| \leq C_2 |\Psi(0;\omega)| \quad \tfa a>0 \text{ and for all }\omega>0.
\eeqs
Furthermore, by using the asymptotics \eqref{eq:9alt} in the definition of $\Psi(a;\omega)$, we see that it is sufficient to prove that, given $\omega_0$ there exists a $C_3(\omega_0)>0$ such that
\beq\label{eq:inter1}
|\Psi(a;\omega) - \Psi(0, \omega)| \leq C_3(\omega_0) \quad \tfa \omega \leq \omega_0,
\eeq
and
\beq\label{eq:inter2}
|\Psi(a;\omega) - \Psi(0, \omega)| = \cO(1/\omega)\quad \text{ when } \omega \tendi.
\eeq

Now, from the definition of $\Psi(a;\omega)$, and using the successive substitutions $\xi = \re^{\ri \pi/4}s + \omega$ and $p=s+ \omega/\sqrt{2}$, we obtain the following estimates:
\begin{align*}
\big| \Psi(a;\omega) - \Psi(0;\omega)\big| &= \left| \int_\omega^{a\re^{\ri\pi/4}+ \omega} \re^{\ri \xi^2} \, \rd \xi\right|
= \left|
\int_0^a \re^{\ri \omega^2}  \re^{-s^2 + 2 \ri \re^{\ri \pi/4}s\omega}\re^{\ri\pi/4} \rd s
\right| 
\leq \int_0^a \re^{-s^2 - \sqrt{2}s\omega}\rd s \\
&= \re^{\omega^2/2} \int_{\omega/\sqrt{2}}^{a+ \omega/\sqrt{2}} \re^{-p^2} \rd p
\leq \re^{\omega^2/2} \int_{\omega/\sqrt{2}}^{\infty} \re^{-p^2} \rd p=
\begin{cases}
\cO(1) & \text{ if } \omega= \cO(1),\\
\cO(1/\omega) & \text{ if } \omega \tendi,
\end{cases}
\end{align*}
where in the last step we have used asymptotics analogous to \eqref{eq:9alt}, which can be obtained by integration by parts. Therefore,  we have proved \eqref{eq:inter1} and \eqref{eq:inter2}, and hence the proof is complete.
\epf

\subsection{Step 4: Estimating the error term}

\ble[Asymptotics of the error term]\label{lem:error}
With $F(u;t,\La)$ defined by \eqref{eq:F} and $\Phi(u;t,\La)$ defined by \eqref{eq:5}, we have the estimate
\beq\label{eq:9a}
\intFokas{0} F'(u;t,\La) \Phi(u;t,\La)\, \rd u = o\Big(\Phi(0;t,\La)\Big) \quad\tas t\tendi,
\eeq
where the $o(\cdot)$ is independent of $\La$.
\ele

We will prove Lemma \ref{lem:error} via the splitting argument in \S\ref{sec:rigor}, but we first show how establishing this result proves Theorem \ref{thm:1}.

\

\bpf[Proof of Theorem \ref{thm:1} assuming Lemma \ref{lem:error}]
Combining \eqref{eq:key} and \eqref{eq:9a} we have
\beqs
\tJ_B = \Phi(0;t,\La) + o\Big( \Phi(0;t.\La)\Big),
\eeqs
and then the results \eqref{eq:result1} and \eqref{eq:result2} follow from using \eqref{eq:7a} and \eqref{eq:7b} respectively.
\epf

\

\noi Therefore, we only need to prove Lemma \ref{lem:error}. The following lemma gives us the required properties of each part of the split integral.

\ble[Properties of each part of the split integral]\label{lem:3.7}

\

\noi 1) Given $\eps>0$, there exists $t_0>0$ and $\delta_1(\eps)>0$
such that
\beq\label{eq:10}
\int_0^{\delta_1(\eps) \re^{\ri \pi/4}} \big| F'(u;t,\La)\big| \big|\rd u\big| \leq \eps
\eeq
for all $t\geq t_0$ and for all $\La$ in the range \eqref{eq:range}.

\noi 2) There exists $t_1>0$, $m>0$, independent of $t$ and $\La$ and $C_2(\delta_1(\eps))>0$,
such that
\beq\label{eq:11}
\int_{\delta_1(\eps)\re^{\ri \pi/4}}^{\infty \re^{\ri \pi/4}}
\exp\left( -\frac{\la_ct}{4} |u|^2\right)\big| F'(u;t,\La)\big| \big| \rd u \big| \leq C_2\big(\delta_1(\eps)\big) \,t^m,
\eeq
for all $t\geq t_1$ and for all $\La$ in the range \eqref{eq:range}.
\ele

\bpf[Proof of Lemma \ref{lem:3.7}]
From the definition of $F(u;t,\La)$ in \eqref{eq:F}, we have
\begin{align}\nonumber
F'\left(u;t,\La\right) = &\diff{^2 \zeta}{u^2}(u) \left( 1- \zeta(u)\right)^{-1/2} \left(1+ \la_c \zeta(u)\right)^{\sigma-1/2} +\half \left( \diff{\zeta}{u}(u)\right)^2 \left(1-\zeta(u)\right)^{-3/2} \left(1+ \la_c \zeta(u)\right)^{\sigma-1/2} \\
&+ \left( \diff{\zeta}{u}(u)\right)^2 \left(1- \zeta(u)\right)^{-1/2}\left(\sigma-\half\right) \left(1+ \la_c \zeta(u)\right)^{\sigma-3/2} \la_c.
\label{eq:12}
\end{align}
1) For the proof of \eqref{eq:10}, observe that it is sufficient to prove that there exists $u_0>0$, $t_0>0$, and $C_0>0$, such that
\beq\label{eq:33}
\big|F'\left(u;t,\La\right)\big| \leq C_0,
\eeq
for all $u$ such that $u= |u|\re^{\ri \pi/4}$ with $0\leq |u|\leq u_0$, for all $t\geq t_0$, and for all $\La$ in the range \eqref{eq:range}.
Indeed, if \eqref{eq:33} holds, then, for $\delta_1<u_0$,
\beqs
\int_0^{\delta_1\re^{\ri \pi/4}} \big| F'\left(u;t,\La\right)\big| \big|\rd u\big|\leq \delta_1 C_0,
\eeqs
and then we let $\delta_1(\eps):= \min(\eps/C_0, u_0)$.

Now, from \eqref{eq:3},
\beqs
\diff{\zeta}{u}(u)= 1 + \frac{(1+ \la_c)u - \ln(1+ \la_c \zeta) + \ln(1-\zeta)}{\ln(1+\La) + \ln(1+ \la_c \zeta) - \ln(1-\zeta)}.
\eeqs
For $|\zeta|$ sufficiently small with $\Re\zeta \geq 0$, we have $\Re\left( \ln(1+ \la_c \zeta) - \ln(1-\zeta)\right)\geq 0$, and so
\beqs
\big| \ln(1+\La)+\ln(1+\la_c \zeta) -\ln(1-\zeta)\big|\geq \big|\ln(1+\la_c \zeta) -\ln(1-\zeta)\big|.
\eeqs
Therefore, for $|u|$ sufficiently small,
\beqs
\left|\diff{\zeta}{u}(u)\right|\leq 1 + \frac{\left|(1+ \la_c)u - \ln(1+ \la_c \zeta) + \ln(1-\zeta)\right|}{\left|\ln(1+ \la_c \zeta) - \ln(1-\zeta)\right|}.
\eeqs
By using
\beq\label{eq:log}
\ln(1+ \xi) = \xi - \half \xi^2 + \cO(\xi^3) \quad\tas \xi\tendo,
\eeq
together with the fact that $\zeta= u + \cO(u^2)$ as $u\tendo$, with the $\cO(\cdot)$ independent of $\La$, from Remark \ref{rem:Taylor}, we have
\begin{align*}
\left|\diff{\zeta}{u}(u) \right| = 1 + \frac{
\left| (1+ \la_c) u - (1+ \la_c)\zeta + \cO(\la_c^2 \zeta^2) + \cO(\zeta^2)\right|
}{
\left| (1+ \la_c) \zeta + \cO(\la_c^2 \zeta^2) + \cO(\zeta^2)\right|
}
= 1 + \frac{
\left| (1+ \la_c) \cO(u^2)\right|
}{
\left| (1+ \la_c)u + \cO(u^2)\right|
},
\end{align*}
where both the $\cO(u^2)$s are uniform in $\La$ and $\la_c$ (and hence $t$) since $\la_c \tendo$ as $t\tendi$.
In other words, we have
\beqs
\left|\diff{\zeta}{u}(u)\right| = 1+ \cO(u) \quad\tas u \tendo,
\eeqs
and a similar calculation shows that
\beqs
\left|\diff{^2 \zeta}{u^2}(u)\right| = \cO(1) \quad \tas u \tendo,
\eeqs
where in both cases the $\cO(\cdot)$ is independent of $\La$ and $t$. Using these asymptotics in \eqref{eq:12}, along with $\zeta= u + \cO(u^2)$ as $u\tendo$ and the fact that
$\la_c \tendo$ as $t\tendi$, we can prove \eqref{eq:33} and thus \eqref{eq:10}.

\

\noi 2) For the proof of \eqref{eq:11}, the result will follow if we can show that there exists $t_1>0$, $m_1>0$, and $m_2>0$
(all independent of $t, \La,$ and $u$) and $C_3(\delta_1(\eps))>0$, 
such that
\beq\label{eq:chip}
\big|F'\left(u;t,\La\right)\big| \leq C_3\left(\delta_1(\eps)\right)\, t^{m_1}\left(1+ |u|\right)^{m_2},
\eeq
for all $u$ such that $u= |u|\re^{\ri \pi/4}$ and $|u|\geq \delta_1(\eps)$ and for all $t\geq t_1$.
Indeed, denoting the integral on the left-hand side of \eqref{eq:11} by $I(t,\La)$ and using \eqref{eq:chip}, we have
\beqs
I(t,\La) \leq C_3\left(\delta_1(\eps)\right)\, t^{m_1}\int_{\delta_1(\eps)\re^{\ri \pi/4}}^{\infty \re^{\ri \pi/4}} \re^{-\frac{\la_c t}{4}|u|^2} \left(1+ |u|\right)^{m_2} |\rd u|.
\eeqs
Since $\la_c t\sim t^{\delta}$ as $t\tendi$ (from the definition \eqref{eq:la_c}), we let $p= u t^{\delta/2}$ and then there exists a $C_4>0$, independent of $t$ and $\La$, such that, for all $t\geq t_1$,
\begin{align*}
I(t;\La)&\leq C_3\left(\delta_1(\eps)\right)\, t^{m_1}\int_{\delta_1(\eps)\re^{\ri \pi/4}t^{\delta/2}}^{\infty \re^{\ri \pi/4}}
 \re^{-C_4|p|^2} \left(1+ \frac{|p|}{t^{\delta/2}}\right)^{m_2} \frac{|\rd p|}{t^{\delta/2}},
\\
& \leq C_3\left(\delta_1(\eps)\right)\, t^{\widetilde{m}}
 \int_0^{\infty} \re^{-C_4|p|^2} (1+ |p|)^{m_2} |\rd p|,
\end{align*}
for some $\widetilde{m}>0$ (dependent on $\delta$, $m_1$, and $m_2$) and the result \eqref{eq:11} follows. Therefore, it is sufficient to prove \eqref{eq:chip}.

We now claim that to prove \eqref{eq:chip} it is sufficient to prove that there exist real ${\cC}_1, \cC_2$, $\alpha_1, \alpha_2, n_1, n_2,$ all independent of $t$ and $\La$ but with $\cC_1, \cC_2>0$ and possibly dependent on $\delta_1(\eps)$,
 such that
\beq\label{eq:s3}
\cC_1 t^{\alpha_1} |u|^{n_1} \leq |\zeta| \leq \cC_2 t^{\alpha_2} |u|^{n_2},
\eeq
for all $u= |u|\re^{\ri\pi/4}$ with $|u|\geq \delta_1(\eps)$, for all $t\geq t_1$, and for all $\La$ in the range \eqref{eq:range}.
Indeed, \eqref{eq:chip} follows from \eqref{eq:s3} and (i) the form of $\rd \zeta/\rd u$ \eqref{eq:3} and subsequent form of $\rd^2 \zeta/\rd u^2$, (ii) the fact that $0\leq \ln(1+\La) \leq C(\delta)\ln t$ from \eqref{eq:range}, (iii) the form of $F'(u;t,\La)$ \eqref{eq:12}, and (iv) the fact that $1/2\leq \sigma\leq 1$ from \eqref{eq:g}.
Therefore, it is sufficient to prove \eqref{eq:s3}.

To prove \eqref{eq:s3}, we first note that the implicit definition of $u$ \eqref{eq:2} implies that, given $\delta_1(\eps)$, there exist $\cC_3, \cC_4>0$, independent of $t$ and $\La$, but dependent on $\delta_1(\eps)$, such that
\beq\label{eq:tricky1}
\cC_3 \la_c |u|^2 \leq \big|f_1\left(\zeta;t,\La\right)\big| \leq \cC_4 \la_c \Big( |u|^2  + \ln t |u|\Big),
\eeq
for all $u= |u|\re^{\ri\pi/4}$ with $|u|\geq \delta_1(\eps)$ and for all $t\geq 1 $.

From the definition of $f_1(\zeta;t,\La)$ \eqref{eq:f1}, we have that given $\zeta_0>0$ and $t_3>0$, there exist $\cC_5, \cC_6>0$, independent of $t$ and $\La$, but dependent on $\zeta_0$, such that
\beq\label{eq:tricky2}
\cC_5 \la_c |\zeta| \Big[ \ln (1+ \La) + 1\Big] \leq \big|f_1\left(\zeta;t,\La\right)\big| \leq \cC_6 \la_c |\zeta| \Big[ \ln(1+\La) + 1 + (1+ \la_c) |\zeta|\Big],
\eeq
for all $\zeta=|\zeta|\re^{\ri \phi}$ with $|\zeta|\geq \zeta_0$ and $\phi$ satisfying \eqref{eq:phi1} and \eqref{eq:phi2}, for all $t\geq t_3$ and for all $\La$ in the range \eqref{eq:range}.

Therefore, choosing $\zeta_0$ to depend on $\delta_1(\eps)$ in such a way that $|\zeta|\geq \zeta_0$ when $u=|u|\re^{\ri \pi/4}$ with $|u|\geq \delta_1(\eps)$, and then combining \eqref{eq:tricky1} and \eqref{eq:tricky2}, we have the required result \eqref{eq:s3}.
\epf

\

\noi With Lemma \ref{lem:3.7} in hand, we can now prove Lemma \ref{lem:error}:

\

\bpf[Proof of Lemma \ref{lem:error}]
Let
\beqs
I_1(t,\La): = \int_0^{\delta_1(\eps)\re^{\ri \pi/4}}F'\left(u;t,\La\right) \Phi\left(u;t,\La\right) \rd u \quad\tand\quad I_2(t,\La):=\intFokas{\delta_1(\eps)\re^{\ri \pi/4}}F'\left(u;t,\La\right) \Phi\left(u;t,\La\right)  \rd u,
\eeqs
so that
\beqs
\intFokas{0}F'\left(u;t,\La\right)\,\Phi\left(u;t,\La\right) \rd u = I_1(t,\La)+ I_2(t,\La).
\eeqs
By Part 1 of Lemma \ref{lem:3.7} and Lemma \ref{lem:3.4}, given $t_0>0$ and $\eps>0$, we have
\beq\label{eq:final1}
|I_1(t,\La)| \leq \eps\,C_1\big| \Phi(0;t,\La) \big|
\eeq
for all $t\geq t_0$ and for all $\La$ in the range \eqref{eq:range}.

Now, the asymptotics \eqref{eq:8a} imply that there exists $t_1>0$ and a $C(\delta(\eps))$ such that
\beqs
\big|\Phi(u;t,\La)\big| \leq C(\delta(\eps)) \exp\left(- \frac{\la_c t}{2} |u|^2\right),
\eeqs
for all $t\geq t_1$, for all $u=|u|\re^{\ri \pi/4}$ with $|u|\geq \delta_1(\eps)$, and for all $\La$ in the range \eqref{eq:range}. Using this bound in the definition of $I_2$, we have
\begin{align}\nonumber
|I_2|
&\leq C(\delta(\eps))\int_{\delta_1(\eps)\re^{\ri \pi/4}}^{\infty \re^{\ri \pi/4}} \big|F'(u)\big| \exp\left(-\frac{\la_c t}{2}|u|^2\right)\big|\rd u\big|,\\ \nonumber
&\leq C(\delta(\eps)) \exp\left(-\frac{\la_c t}{4}|\delta_1(\eps)|^2\right) \int_{\delta_1(\eps)\re^{\ri \pi/4}}^{\infty \re^{\ri \pi/4}} \big|F'(u)\big|\exp\left(-\frac{\la_c t}{4}|u|^2\right)\big|\rd u\big|,\\ \label{eq:final2a}
& \leq C(\delta(\eps))  \exp\left(-\frac{\la_c t}{4}|\delta_1(\eps)|^2\right)C_2\left(\delta,\delta_1(\eps)\right) \,t^m \quad\text{ by Part 2 of Lemma \ref{lem:3.7}}.
\end{align}
The asymptotics of $\Phi(0;t,\La)$, \eqref{eq:7a} and \eqref{eq:7b}, and the range of $\La$ \eqref{eq:range} imply that there exist $C_3, C_4>0$ such that
\beq\label{eq:s2}
C_3 \frac{1}{\la_c t} \frac{1}{\ln t} \leq \big|\Phi(0;t,\La)\big|\leq C_4 \sqrt{\frac{1}{\la_c t}},
\eeq
for all $t\geq t_1$ and for all $\La$ in the range \eqref{eq:range}.

Now, from \eqref{eq:final2a},
\beqs
|I_2(t,\La)| \leq \eps \big|\Phi(0;t,\La)\big|\left(
\frac{1}{\eps}\frac{1}{\big|\Phi(0;t,\La)\big|}  C\left(\delta(\eps)\right)  \exp\left(-\frac{\la_c t}{4}|\delta_1(\eps)|^2\right)C_2\left(\delta,\delta_1(\eps)\right) t^m\right),
\eeqs
and then using \eqref{eq:s2} we see that
\beqs
|I_2| \leq \eps \big|\Phi(0;t,\La)\big|\left(
\frac{1}{\eps}\frac{1}{C_3} C_2\left(\delta,\delta_1(\eps)\right) C\left(\delta(\eps)\right) (\la_c t)^{}(\ln t) t^{m}  \exp\left(-\frac{\la_c t}{4}|\delta_1(\eps)|^2\right)\right),
\eeqs
for all $t\geq t_1$.
Recalling that $\la_c t\sim t^{\delta}$ as $t\tendi$, we see that there exists a $t_2>0$ such that
\beq\label{eq:final2b}
|I_2(t;\La) | \leq \eps \big|\Phi\left(0;t,\La\right)\big| \quad\tfa t\geq t_2.
\eeq
Combining \eqref{eq:final1} and \eqref{eq:final2b} we obtain the result \eqref{eq:9a}.
\epf

\section{Proof of Theorem \ref{JB1+2-theorem}}
\label{sec:Arran}

\subsection{Summary of the method}

The main idea is to split the integral $J_B$ into two parts: $J_{B1}$, an integral along a finite contour that is both real and (when $t$ is large) vanishingly small, and $J_{B2}$, an integral along an infinite contour that is controllably ``far" from the endpoint (and hence from the stationary point). The large-$t$ asymptotics of $J_{B2}$ can be computed to all orders, while those of $J_{B1}$ can be computed at least to first order. This will be sufficient to find the asymptotics of $J_B$ to all orders, since the error term in the asymptotic expression for $J_{B1}$ can be made arbitrarily small compared to the asymptotics of $J_{B2}$ by making an appropriate choice of the splitting point for the integrals.

The method can be summarised as follows:

\bit
\item Step 1: split the contour of integration into two parts, one small and one infinitely long.

\item Step 2: estimate the infinite-contour integral $J_{B2}$ using an integration by parts argument.

\item Step 3: estimate the small-contour integral $J_{B1}$ by using a truncated Taylor series (i.e. a local expansion).

\item Step 4: choose the splitting point for the contour so as to appropriately bound the remainder terms from both the previous steps.
\eit

\subsection{Step 1: Splitting the integral}

We define $J_{B1}$ and $J_{B2}$ as follows:
\begin{align}
\label{JB1-defn}
J_{B1}(t;\la):=\int_{1-t^{\delta-1}}^{1-k}(1-z)^{-1/2}\re^{\ri tF(z;\lambda)}\,\mathrm{d}z \quad\tand\quad
J_{B2}(t;\la):=\int_{1-k}^{\infty \re^{\ri \phi}}(1-z)^{-1/2}\re^{\ri tF(z;\lambda)}\,\mathrm{d}z,
\end{align}
where $k=k(t,\delta)$ is chosen so that
\beq\label{eq:k}
0<k<t^{\delta-1}.
\eeq
We will fix $k$ as a specific function of $t$ and $\delta$ in Step 4. By Cauchy's theorem, we have
 \beq\label{eq:mark}
 J_B(t;\la)=J_{B1}(t;\la)+J_{B2}(t;\la).
 \eeq

\subsection{Step 2: The asymptotics of $J_{B2}$}
We now prove two lemmas about the behaviour of the phase function $F(z;\la)$.

\ble
\label{JB2-lemma1}
When $z$ is on the contour
\beq\label{eq:contour}
\{z: z=1-k+R\re^{\ri \phi}, \quad0\leq R<\infty\},
\eeq
we have
\begin{equation}
\label{dFdz-estimate}
\bigg|\frac{\partial F}{\partial z}\bigg|>\min\left(\frac{\pi}{2}-\phi,\ln\left(\frac{t^{\delta-1}}{k}\right)\right),
\end{equation}
and for sufficiently large $R$, $\big|\partial F/\partial z\big|> \pi/2-\phi$.
\ele

\begin{proof}
We split $\partial F/\partial z$ into its real and imaginary parts, as follows: 
\beq\label{eq:dFdz}
\frac{\partial F}{\partial z}=\ln\left(\frac{z}{1-z}\right)+\ln\lambda=\ln\Big|\frac{z}{1-z}\Big|+\ln\lambda+\ri \arg\left(\frac{z}{1-z}\right).
\eeq
As $z$ moves along the contour \eqref{eq:contour} with $R$ increasing, $\arg(z)$ is strictly increasing from $0$ towards $\phi$, and $\arg(1-z)$ is strictly decreasing from $0$ towards $\phi-\pi$, so the imaginary part $\arg\left(\frac{z}{1-z}\right)$ increases monotonically from $0$ towards $\pi$.

\begin{figure}
  \centering 
 {
    \includegraphics[width=0.8\textwidth]{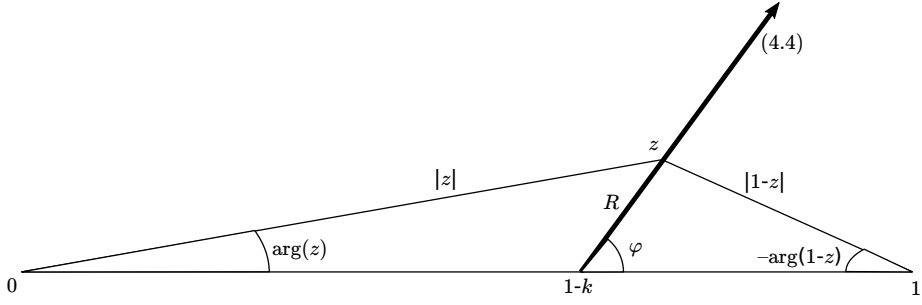}
  }
  \caption{The geometry of the integration contour \eqref{eq:contour}, for small $R$}\label{fig:1}
\end{figure}

For very small $R>0$, we can see that $|z|$ is increasing and $|1-z|$ is decreasing (see Figure \ref{fig:1}), so $\big|\frac{z}{1-z}\big|$ is increasing from its initial value of $\frac{1-k}{k}\sim k^{-1}\gg1$ ($t$ is large and $\delta>0$ is small, and thus $k<t^{\delta-1}$ is small). But for very large $R$, it is clear that $\big|\frac{z}{1-z}\big|\sim1$. So at some point the function $\big|\frac{z}{1-z}\big|$ must stop increasing and start decreasing, i.e. its derivative must change sign. We now find out where this point is by considering the simpler function $\big|\frac{z}{1-z}\big|^2$.

Using the parametrisation in \eqref{eq:contour}, we write all the relevant functions in terms of $R$ and not $z$:
\begin{align}\nonumber
&1-z=k-R\cos\phi-\ri R\sin\phi; \\ 
\label{eq:square} |z|^2=(1-k)^2+2(1-k)&R\cos\phi+R^2; \qquad |1-z|^2=k^2-2kR\cos\phi+R^2.
\end{align}
Now we can compute the value of $R$ at which the derivative is zero:
\begin{align*}
\frac{\partial}{\partial R}\bigg(\Big|\frac{z}{1-z}\Big|^2\bigg)=0&\Leftrightarrow\big|1-z\big|^2\frac{\partial}{\partial R}\left(|z|^2\right)-|z|^2\frac{\partial}{\partial R}\left(\big|1-z\big|^2\right)=0 \\
\begin{split}
&\Leftrightarrow\left(k^2-2kR\cos\phi+R^2\right)\left(2(1-k)\cos\phi+2R\right) \\
&\hspace{2cm}-\left((1-k)^2+2(1-k)R\cos\phi+R^2\right)\left(-2k\cos\phi+2R\right)=0
\end{split} 
\begin{split}
\end{split} \\
&\Leftrightarrow\big[-\cos\phi\big]R^2+\big[2k-1\big]R+\big[k(1-k)\cos\phi\big]=0 \\
&\Leftrightarrow R=\frac{1-2k\pm\sqrt{1-4k\sin^2\phi+4k^2\sin^2\phi}}{-2\cos\phi}.
\end{align*}
Since $k$ is small, we find the following first-order approximation for the critical value of $R$: \[R=\frac{1-2k\pm(1-4k\sin^2\phi+4k^2\sin^2\phi)^{1/2}}{-2\cos\phi}\approx\frac{1-2k\pm(1-2k\sin^2\phi)}{-2\cos\phi}.\] Taking the positive sign gives a negative value of $R$, so we take the negative sign and obtain \[R\approx\frac{1-2k-1+2k\sin^2\phi}{-2\cos\phi}=k\cos\phi.\]

Thus, we have proved that the function $\big|\frac{z}{1-z}\big|$ has a single stationary point for $R\geq0$, namely a maximum at a value of $R$ somewhere close to $k\cos\phi$. We now find the maximal value of $\big|\frac{z}{1-z}\big|$ by evaluating this function at $R=k\cos\phi$. 
With this value of $R$, using the above formulae, we have:
\begin{align*}
|z|^2&=(1-k)^2+2(1-k)k\cos^2\phi+k^2\cos^2\phi=1-2k\sin^2\phi+k^2\sin^2\phi, \\
\big|1-z\big|^2&=k^2-2k^2\cos^2\phi+k^2\cos^2\phi=k^2\sin^2\phi, \\
\Big|\frac{z}{1-z}\Big|^2&=\frac{1}{k^2\sin^2\phi}-\frac{2}{k}+1\sim\frac{1}{k^2\sin^2\phi},
\end{align*}
and thus
\beqs
\Big|\frac{z}{1-z}\Big|\sim\frac{1}{k\sin\phi}>k^{-1}.
\eeqs

We are now in a position to estimate $\partial F/\partial z$, by bounding either its real part or its imaginary part according to the value of $R$. We split into two cases as follows.

\textbf{Case 1:} $\boldsymbol{R\geq k\cos\phi.}$ Here we consider the imaginary part, namely $\arg\left(\frac{z}{1-z}\right)$, which we know is monotonically increasing and therefore bounded below by its value at $R=k\cos\phi$. We can see from Figure \ref{fig:2} that $R=k\cos\phi$ gives $\arg(1-z)=-(\frac{\pi}{2}-\phi)$ and therefore $\arg\big(\frac{z}{1-z}\big)>\pi/2-\phi$. 
Thus, in this case we have:
\begin{equation}
\label{dFdz-estimate1}
\bigg|\frac{\partial F}{\partial z}\bigg|>\frac{\pi}{2}-\phi.
\end{equation}

\begin{figure}
  \centering
  {
    \includegraphics[width=0.8\textwidth]{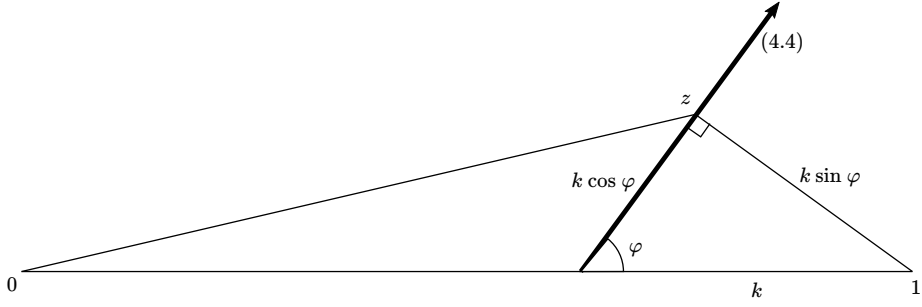}
  }
  \caption{The geometry of the integration contour in \eqref{eq:contour} in the case when $R=k\cos\phi$}\label{fig:2}
  \end{figure}

\textbf{Case 2:} $\boldsymbol{R\leq k\cos\phi.}$ Here we consider the real part, namely $\ln\big|\frac{z}{1-z}\big|+\ln\lambda$. We know this quantity is monotonically increasing up to \textit{approximately} $R=k\cos\phi$, and that it is greater than its initial value \textit{at} $R=k\cos\phi$, so it must be bounded below by its initial value, i.e. by $\ln\left(\frac{(1-k)\la}{k}\right)$. Using the lower bound on $\la$ in \eqref{eq:3old}, as well as the assumption $k<t^{\delta-1}$ from \eqref{eq:k}, we then have
\begin{align}
\nonumber
\bigg|\frac{\partial F}{\partial z}\bigg|&\geq\Big|\ln\left(\frac{1-k}{k}\right)+\ln\lambda\Big|>\Big|\ln\left(\frac{1}{k}-1\right)-\ln\left(t^{1-\delta}-1\right)\Big| \\
\label{dFdz-estimate2}
&=\ln\left(\frac{t^{\delta-1}}{k}\right)+\ln(1-k)+\ln\left(1-t^{\delta-1}\right)>\ln\left(\frac{t^{\delta-1}}{k}\right).
\end{align}
Putting together the estimates (\ref{dFdz-estimate1}) and (\ref{dFdz-estimate2}), we have the desired result (\ref{dFdz-estimate}).
\end{proof}

\

Motivated by the results of Lemma \ref{JB2-lemma1}, we introduce the following notation, which will make some of the later calculations simpler.

\begin{definition}
Let
\begin{equation}
\label{d-defn}
D=D(\lambda):=\frac{\partial F}{\partial z}(1-k;\la)=\ln\left(\frac{(1-k)\lambda}{k}\right),
\end{equation}
and let
\begin{equation*}
D_-:=\ln\left(\frac{t^{\delta-1}}{k}\right).
\end{equation*}
\end{definition}

\begin{remark}
We saw in (\ref{dFdz-estimate2}) that $D$ always has $D_-$ as a lower bound, but it can only be close to this value if $\lambda$ is close to its critical value $\lambda_c$. In general $D$ might be as large as $\cO(\ln t)$.
\end{remark}

As a corollary of Lemma \ref{JB2-lemma1}, we can identify a particular situation where $D_-$ is a lower bound for $|\partial F/\partial z|$ on the whole of the contour \eqref{eq:contour} (i.e.~not just at the endpoint $z=1-k$).

\begin{corollary}\label{cor:phase}
Suppose that $k$ is close enough to $t^{\delta-1}$ that $D_-\ll1$,
and define $\phi$ as follows:
\beq\label{eq:phi_new}
\phi=\frac{\pi}{4} \text{ when } \ln \la \geq 0, \quad \tand\quad \phi =\half \arctan\left(\frac{\pi}{|\ln\la|}\right) \text{ when } \ln\la<0;
\eeq
note that these choices satisfy the required conditions \eqref{eq:phi1} and \eqref{eq:phi2} on $\phi$.

Then
\beqs
\left|\pdiff{F}{z}(z;\la)\right| \geq D_-
\eeqs
for all $z$ on the contour \eqref{eq:contour}.
\end{corollary}

\begin{proof}
The definition (\ref{eq:phi_new}) implies that $\phi\leq\pi/4$ regardless of $\lambda$, so $\pi/2-\phi\geq\pi/4$. We also have $D_-\ll1$, so the result of Lemma \ref{JB2-lemma1} becomes
\begin{equation*}
\bigg|\frac{\partial F}{\partial z}(z,\lambda)\bigg|>\min\left(\frac{\pi}{2}-\phi,D_-\right)=D_-,
\end{equation*}
as required.
\end{proof}

\ble
\label{JB2-lemma2}
When $z$ is on the contour \eqref{eq:contour}, the function $F(z;\lambda)$ always has non-negative imaginary part.
\ele

\begin{proof}
Clearly $\mathrm{Im}(F)=0$ when $R=0$, since then $z$ and $\lambda$ are both real.

When $R>0$ is very small, we can estimate $F$ as follows:
\begin{align*}
F(z;\lambda)&=(k-R\re^{\ri \phi})\ln(k-R\re^{\ri \phi})+(1-k+R\re^{\ri \phi})\ln(1-k+R\re^{\ri \phi})+(1-k+R\re^{\ri \phi})\ln\lambda, \\
\begin{split}
&=(k-R\re^{\ri \phi})\left(\ln k+\ln\left(1-\frac{R\re^{\ri \phi}}{k}\right)\right) \\
&\hspace{2cm}+(1-k+R\re^{\ri \phi})\left(\ln(1-k)+\ln\left(1+\frac{R\re^{\ri \phi}}{1-k}\right)\right)+(1-k+R\re^{\ri \phi})\ln\lambda,
\end{split} \\
&\sim(k-R\re^{\ri \phi})\left(\ln k-\frac{R\re^{\ri \phi}}{k}\right)+(1-k+R\re^{\ri \phi})\left(\ln(1-k)+\frac{R\re^{\ri \phi}}{1-k}\right)+(1-k+R\re^{\ri \phi})\ln\lambda, \\
&\sim\big[k\ln k+(1-k)\ln(1-k)+(1-k)\ln\lambda\big]+R\re^{\ri \phi}\big[-\ln k-1+\ln(1-k)+1+\ln\lambda\big].
\end{align*}
Therefore,
\begin{equation*}
\mathrm{Im}(F)\sim R\sin\phi\Big[-\ln k+\ln(1-k)+\ln\lambda\Big]=R\sin\phi\ln\left(\lambda\left(\frac{1}{k}-1\right)\right).
\end{equation*}
Thus, since $k<t^{\delta-1}$ and $\lambda\geq\frac{t^{\delta-1}}{1-t^{\delta-1}}$, we have $\mathrm{Im}(F)>0$ for $R>0$ small.

Clearly, $F(z;\lambda)$ is an analytic function of $z$ for $\mathrm{Im}(z)>0$, so 
\begin{align*}
\frac{\partial}{\partial R}\Big(\mathrm{Im}\,F\left(1-k+R\re^{\ri \phi},\lambda\right)\Big)&=\frac{\partial}{\partial z}\Big(\mathrm{Im}\,F(z;\lambda)\Big)\bigg|_{z=1-k+R\re^{\ri \phi}}=\mathrm{Im}\left(\frac{\partial F}{\partial z}\right)\bigg|_{z=1-k+R\re^{\ri \phi}} \\
&=\Big(\arg(z)-\arg(1-z)\Big)\bigg|_{z=1-k+R\re^{\ri \phi}}>0.
\end{align*}
Thus, $\mathrm{Im}(F)$ is strictly increasing along the contour, which means it must be positive for all $R>0$, as required.
\end{proof}

\

We now prove the following lemma which allows us to simplify the terms arising from repeated integration by parts.

\ble
\label{JB2-lemma3}
For any $N\in\mathbb{N}$,
\begin{equation}
\label{JB2-IbP-simplification}
\bigg(\frac{\partial}{\partial z}\cdot\frac{-1}{\ri t\frac{\partial F}{\partial z}}\bigg)^N\left((1-z)^{-1/2}\right)=\frac{(1-z)^{-(2N+1)/2}}{(-it)^N\left(\partial F/\partial z\right)^{N}}\sum_{m,n=0}^NA_{mn}z^{-m}\left(\frac{\partial F}{\partial z}\right)^{-n},
\end{equation}
where $F$ is defined by \eqref{eq:2old} and the $A_{mn}$ are dyadic rationals satisfying $|A_{mn}|<(3N)!$ for all $m,n$ and $A_{mN}=0$ for all $m<N$, $A_{NN}=(-1)^N(2N-1)!!:=(2N-1)(2N-3)\dots(5)(3)(1)$.
\ele

\begin{proof}
Firstly, the expression for $\partial F/\partial z$, \eqref{eq:dFdz}, implies that
\[\frac{\partial^2F}{\partial z^2}=\frac{1}{z}+\frac{1}{1-z}=\frac{1}{z(1-z)}.\] So for $N=1$ we have
\begin{align*}
\frac{\partial}{\partial z}\bigg(\frac{(1-z)^{-1/2}}{-it\frac{\partial F}{\partial z}}\bigg)
=\frac{(1-z)^{-3/2}\left(\frac{1}{2}\cdot\frac{\partial F}{\partial z}+(z-1)\frac{\partial^2F}{\partial z^2}\right)}{-it\left(\partial F/\partial z\right)^2} =\frac{(1-z)^{-3/2}}{-it\left(\partial F/\partial z\right)}\bigg[\frac{1}{2}-z^{-1}\left(\frac{\partial F}{\partial z}\right)^{-1}\bigg],
\end{align*}
which is in the required form. For $N=2$, similar calculations give that
\begin{align*}
\frac{\partial}{\partial z}\Bigg(\frac{(1-z)^{-3/2}\left(\frac{1}{2}\cdot\frac{\partial F}{\partial z}-\frac{1}{z}\right)}{-t^2\left(\partial F/\partial z\right)^3}\Bigg) 
=\frac{(1-z)^{-5/2}}{-t^2\left(\partial F/\partial z\right)^2}\bigg[\frac{3}{4}-\left(\frac{7}{2}z^{-1}-z^{-2}\right)\left(\frac{\partial F}{\partial z}\right)^{-1}+3z^{-2}\left(\frac{\partial F}{\partial z}\right)^{-2}\bigg],
\end{align*}
which is again in the required form.

For the general case, we proceed by induction. Assuming the equation (\ref{JB2-IbP-simplification}) is valid with $N$ replaced by $N-1$, and differentiating as before, we find that the LHS of (\ref{JB2-IbP-simplification}) is:
\begin{align*}
&\bigg(\frac{\partial}{\partial z}\cdot\frac{-1}{\ri t\frac{\partial F}{\partial z}}\bigg)^N\left((1-z)^{-1/2}\right) =\frac{\partial}{\partial z}\Bigg[\frac{(1-z)^{-(2N-1)/2}}{(-it)^N\left(\partial F/\partial z\right)^{N}}\sum_{m,n=0}^{N-1}A_{mn}z^{-m}\left(\frac{\partial F}{\partial z}\right)^{-n}\Bigg],\\
\,&=\sum_{m,n=0}^{N-1}\frac{A_{mn}(1-z)^{-(2N+1)/2}}{(-it)^N\left(\partial F/\partial z\right)^N}\bigg[\left(N+m-\frac{1}{2}\right)z^{-m}\left(\frac{\partial F}{\partial z}\right)^{-n}\\
&\hspace{5.5cm}-
mz^{-m-1}\left(\frac{\partial F}{\partial z}\right)^{-n}-(N+n)z^{-m-1}\left(\frac{\partial F}{\partial z}\right)^{-n-1}\bigg],
\end{align*}
which can be rearranged to an expression in the required form. Note that the only term with $n=N$ is the one with $m=n=N$, by the inductive hypothesis.
\end{proof}

Given the three lemmas above, we are now in a position to compute the large-$t$ asymptotics of $J_{B2}$.

\ble[Large-$t$ asymptotic expansion of $J_{B2}$]
\label{JB2-theorem}
We have
\begin{equation}
\label{JB2-asymptotics}
J_{B2}(t;\la)=\sum_{j=1}^{N-1}T_j(t;\la)+R_N(t;\la),
\end{equation}
where the terms $T_j$ are defined by 
\beq\label{eq:Tj}
T_j(t;\la)=-\bigg(\frac{-1}{\ri t\frac{\partial F}{\partial z}}\cdot\frac{\partial}{\partial z}\bigg)^j\bigg(\frac{(1-z)^{-1/2}}{\ri t\frac{\partial F}{\partial z}}\bigg)\re^{\ri tF(z;\lambda)}\Bigg|_{z=1-k}
\eeq
and the remainder term $R_N$ is defined by
\beqs
R_N(t;\la)=\int_{1-k}^{\infty \re^{\ri \phi}}\bigg(\frac{\partial}{\partial z}\cdot\frac{-1}{\ri t\frac{\partial F}{\partial z}}\bigg)^N\left((1-z)^{-1/2}\right)\re^{\ri tF(z;\lambda)}\,\mathrm{d}z. 
\eeqs
If $\phi$ is defined by \eqref{eq:phi_new} and $k$ satisfies
\begin{equation}
\label{JB2-k-assumption}
(kt)^{\epsilon-\frac{m}{2m+1}}\ll D_-\ll1
\end{equation}
for some $m\in\mathbb{N}$ and $\epsilon>0$, then:
\begin{align}
T_j(t;\la)&=\cO\left((2j-1)!! t^{-j-1}k^{-(2j+1)/2}D_-^{-2j-1}\right); \label{Tj-estimate} \\
R_N(t;\la)&=\cO\left((2N-1)!! t^{-N}(\ln t)^{(2N+1)/2}k^{-(2N-1)/2}D_-^{-2N}\right);\label{RN-estimate}
\end{align}
and thus (\ref{JB2-asymptotics}) provides a valid large-$t$ asymptotic expansion.
\ele

\begin{proof}
We apply integration by parts $N$ times to the definition (\ref{JB1-defn}) of $J_{B2}$ to obtain
\begin{multline}
\label{JB2-IbP}
J_{B2}(t;\la)=\sum_{j=0}^{N-1}\Bigg[\bigg(\frac{-1}{\ri t\frac{\partial F}{\partial z}}\cdot\frac{\partial}{\partial z}\bigg)^j\bigg(\frac{(1-z)^{-1/2}}{\ri t\frac{\partial F}{\partial z}}\bigg)\re^{\ri tF(z;\lambda)}\Bigg]_{1-k}^{\infty \re^{\ri \phi}} \\ +\int_{1-k}^{\infty \re^{\ri \phi}}\bigg(\frac{\partial}{\partial z}\cdot\frac{-1}{\ri t\frac{\partial F}{\partial z}}\bigg)^N\left((1-z)^{-1/2}\right)\re^{\ri tF(z;\lambda)}\,\mathrm{d}z,
\end{multline}
for any $N\in \mathbb{N}$. By Lemmas \ref{JB2-lemma1} and \ref{JB2-lemma2} applied to the series expression given by Lemma \ref{JB2-lemma3}, the $\infty \re^{\ri \phi}$ parts of the boundary terms contribute nothing, and thus \eqref{JB2-IbP} becomes \eqref{JB2-asymptotics}.

We now concentrate on proving the bounds \eqref{Tj-estimate} and \eqref{RN-estimate}.
By Lemma \ref{JB2-lemma2}, $\re^{\ri tF}$ is bounded above by $1$. 
Using this fact, along with Corollary \ref{cor:phase} and Lemma \ref{JB2-lemma3}, we find:
\begin{align}
\nonumber T_j(t;\la)&=\Bigg[\frac{-(1-z)^{-(2j+1)/2}}{(-it)^{j+1}\left(\partial F/\partial z\right)^{j+1}}\sum_{m,n=0}^j\Big[A_{mn}z^{-m}\left(\frac{\partial F}{\partial z}\right)^{-n}\Big]\re^{\ri tF(z;\lambda)}\Bigg]_{z=1-k}, \\
\nonumber
&=\frac{-k^{-(2j+1)/2}}{(-it)^{j+1}D^{j+1}}\sum_{m,n=0}^j\Big[A_{mn}(1-k)^{-m}D^{-n}\Big]\re^{\ri tF(1-k;\lambda)}, \\
\nonumber &=
\cO\left((2j-1)!! t^{-j-1}k^{-(2j+1)/2}\sum_{n=0}^jD_-^{-j-n-1}\right),
\end{align}
which gives the required expression (\ref{Tj-estimate}), since by (\ref{JB2-k-assumption}) we are assuming that $D_-\ll1$.

Using Lemma \ref{JB2-lemma3} again, we have
\beqs
R_N(t;\la)
=\int_{1-k}^{\infty \re^{\ri \phi}}
\Bigg[\frac{(1-z)^{-(2N+1)/2}}{(-it)^N\left(\partial F/\partial z\right)^{N}}\sum_{m,n=0}^NA_{mn}z^{-m}\left(\frac{\partial F}{\partial z}\right)^{-n}\Bigg]\re^{\ri t F(z;\la)}\,\mathrm{d}z. 
\eeqs
From the definition \eqref{eq:contour} of the contour of integration, we have $|z|\geq 1-k$, and then, since $k=o(1)$ as $t\tendi$ (from \eqref{eq:k}), we have $|z|\geq 1/2$, say, for $t$ sufficiently large. Using this fact, along with Corollary \ref{cor:phase}, Lemma \ref{JB2-lemma2}, and the estimates from Lemma \ref{JB2-lemma3}, we have:
\begin{align*}
R_N(t;\la)&=\cO\Bigg(\int_{1-k}^{\infty \re^{\ri \phi}}\frac{(1-z)^{-(2N+1)/2}}{t^N}\sum_{m,n=0}^N\big|A_{mn}\big|\bigg|\frac{\partial F}{\partial z}\bigg|^{-N-n}\,\mathrm{d}z\Bigg), \\
&=\cO\Bigg(\frac{(2N-1)!!}{t^N}D_-^{-2N}\int_{1-k}^{\infty \re^{\ri \phi}}\Big|(1-z)^{-(2N+1)/2}\Big|\,\mathrm{d}z\Bigg).
\end{align*}
Here again we have implicitly used the assumption that $D_-\ll1$ from (\ref{JB2-k-assumption}). Now, from \eqref{eq:square},
\beqs
|1-z|^2 = (k-R)^2\cos\phi + (k^2 +R^2)(1-\cos\phi) \geq \frac{(k+R)^2}{2}(1-\cos\phi),
\eeqs
where we have used the inequality $(k+R)^2 \leq 2(k^2+ R^2)$. Therefore,
\beqs
\frac{1}{|1-z|^{(2N+1)/2}} \leq \left(\frac{2}{1-\cos\phi}\right)^{(2N+1)/4} \frac{1}{(k+R)^{(2N+1)/2}}.
\eeqs

For $\ln\la\geq0$, the first equation in \eqref{eq:phi_new} implies that $(1-\cos\phi)^{-1} =\cO(1)$. For $\ln\la <0$, the second equation in \eqref{eq:phi_new} implies that 
\begin{align*}
\tan2\phi=\frac{\pi}{|\ln\la|},\quad\cos2\phi&=\left(1+\tan^22\phi\right)^{-1/2}=\left(1+\frac{\pi^2}{|\ln\la|^2}\right)^{-1/2},
\end{align*}
and
\begin{align}\label{eq:CE1}
\left(1-\cos\phi\right)^{-1}&=\left(1-\sqrt{\half\left(1+\left(1+\frac{\pi^2}{|\ln\la|^2}\right)^{-1/2}\right)}\;\right)^{-1}.
\end{align}
Therefore, as $|\ln\la|$ increases, $\tan2\phi$ decreases, $\cos2\phi$ increases, and $\left(1-\cos\phi\right)^{-1}$ increases. Thus, the maximal value of $\left(1-\cos\phi\right)^{-1}$ is achieved when $\la$ is minimal, i.e.~when $\la= (t^{1-\delta}-1)^{-1}$. Substituting this into \eqref{eq:CE1} we find that
\begin{align*}
\left(1-\cos\phi\right)^{-1}\sim\frac{8|\ln\la|^2}{\pi^2}=\cO\big((\ln t)^2\big).
\end{align*}
Therefore, in general we have \[|1-z|^{-(2N+1)/2}=\cO\left((\ln t)^{(2N+1)/2}(k+R)^{-(2N+1)/2}\right),\]
and then
\begin{align*}
R_N(t;\la)&=\cO\Bigg(\frac{(2N-1)!!}{t^N}D_-^{-2N}\int_{0}^{\infty}(\ln t)^{(2N+1)/2}(k+R)^{-(2N+1)/2}\,\mathrm{d}R\Bigg), \\
&=\cO\Bigg((2N-1)!! t^{-N}(\ln t)^{(2N+1)/2}D_-^{-2N}\Big[(k+R)^{-(2N-1)/2}\Big]_{R=0}^{\infty}\Bigg), \\
&=\cO\left((2N-1)!! t^{-N}(\ln t)^{(2N+1)/2}D_-^{-2N}k^{-(2N-1)/2}\right),
\end{align*}
as required. 

Finally, it remains to check that (\ref{JB2-asymptotics}) with the estimates (\ref{Tj-estimate}) and (\ref{RN-estimate}) actually gives a valid asymptotic formula, i.e. that each term in the series is smaller than the next term for sufficiently large $t$. From (\ref{Tj-estimate}) we see that the estimate for $T_{j+1}$ is smaller than the one for $T_j$ if and only if \[t^{-1}k^{-1}D_-^{-2}\ll1,\] i.e. if and only if $D_-\gg(kt)^{-1/2},$ which is true by the first half of (\ref{JB2-k-assumption}). Furthermore, from (\ref{Tj-estimate}) and (\ref{RN-estimate}) we see that the estimate for $R_N$ is smaller than the one for \textit{some} $T_j$ (not necessarily $T_{N-1}$) if and only if 
\[t^{-N+j+1}(\ln t)^{(2N+1)/2}k^{-N+j+1}D_-^{-2N+2j+1}\ll1,\] 
which is equivalent to the first half of (\ref{JB2-k-assumption}) with $m=N-j-1$, since we know $k$ behaves like a power of $t$ by the second half of (\ref{JB2-k-assumption}). (Note that if $D_-\gg(kt)^{\epsilon-\frac{m}{2m+1}}$ holds for some $m$, then it also holds for any larger value of $m$, since $kt\gg1$ by (\ref{JB2-k-assumption}), so there is no need to assume $m<N$ in the statement of the theorem.)
\end{proof}

\subsection{Step 3: The asymptotics of $J_{B1}$}

\ble[Large-$t$ asymptotic expansion of $J_{B1}$]
\label{JB1-theorem}
Let
$a:= 1- k \,t^{1-\delta},$
so that
\beq\label{a-defn}
k=t^{\delta-1}(1-a),
\eeq
and assume that this new variable $a$ satisfies
\begin{equation}
\label{JB1-a-assumption}
t^{-\delta/2}\ll a\ll t^{-\delta/3}.
\end{equation}
Then, the large-$t$ asymptotic expansion of $J_{B1}$ is given to first order by
\begin{equation}
\label{JB1-asymptotics}
J_{B1}(t;\la)=
\exp\Big(\ri tF(1-t^{\delta-1};\lambda)-
\ri\omega^2\Big)\,t^{-1/2}\sqrt{\frac{2}{1+\lambda_c}}\left(\int_{\omega}^{\omega + a \sqrt{\la_c t/2}}\re^{\ri\xi^2}\,\mathrm{d}\xi\right)
+\cO\left(t^{-\frac{1}{2}+\frac{3\delta}{2}}a^4\right),
\end{equation}
where $\omega$ is defined by \eqref{eq:omega}.
\ele

\begin{proof}
We start by making two changes of variable in the expression (\ref{JB1-defn}) for $J_{B1}$, in order to improve the notation. Firstly, substituting $x=1-z$ yields the simpler expression
\begin{equation*}
J_{B1}(t;\la)=\int_{k}^{t^{\delta-1}}x^{-1/2}\exp\big(\ri tF(1-x,\lambda)\big)\,\mathrm{d}x.
\end{equation*}
Then, in order to have the critical value at $0$, we substitute $x=t^{\delta-1}(1-\zeta)$ (note that this $\zeta$ variable is not connected to the $\zeta$ variable in \eqref{eq:0} and \S\ref{sec:outline}-\ref{sec:proof}). Thus $\mathrm{d}x/\mathrm{d}\zeta=-t^{\delta-1}$, and the value $x=t^{\delta-1}$ corresponds to $\zeta=0$ as desired, while the value $x=k$ corresponds to $\zeta=1-kt^{1-\delta}<1$. We get:
\beq\label{JB1-x}
J_{B1}(t;\la) =  t^{(\delta-1)/2} \int_0^{1-kt^{1-\delta}}(1-\zeta)^{-1/2} \exp\Big( \ri t F\left(1-t^{\delta-1}(1-\zeta); \la\right)\Big) \rd \zeta.
\eeq

We now expand the exponent in powers of $\zeta$ and then show that the higher powers can be discarded without affecting the leading-order asymptotics of $J_{B1}$. Indeed, expanding the exponent, and using the definition \eqref{eq:la_c} of $\la_c$, we have
\begin{align}\nonumber
\begin{split}
F(1-t^{\delta-1}(1-\zeta),\lambda)&=\left(1-t^{\delta-1}(1-\zeta)\right)\ln\left(1-t^{\delta-1}(1-\zeta)\right)+t^{\delta-1}(1-\zeta)\ln \left(t^{\delta-1}(1-\zeta)\right)
\\
&\hspace{7cm}
+\left(1-t^{\delta-1}(1-\zeta)\right)\ln\lambda ,
\end{split}\\ \nonumber
\begin{split}
&=\left(1-t^{\delta-1}+t^{\delta-1}\zeta\right)\ln\left(1-t^{\delta-1}+t^{\delta-1}\zeta\right)+t^{\delta-1}(1-\zeta)\ln\left(t^{\delta-1}(1-\zeta)\right) \\
&\hspace{7cm}+\left(1-t^{\delta-1}+t^{\delta-1}\zeta\right)\ln\lambda,
\end{split} \\ \nonumber
\begin{split}
&=\left(1-t^{\delta-1}+t^{\delta-1}\zeta\right)\Big[\ln\left(1-t^{\delta-1}\right)+\ln(1+\lambda_c\zeta)\Big]\\
&\hspace{2cm}
+t^{\delta-1}(1-\zeta)\Big[\ln\left(t^{\delta-1}\right)+\ln(1-\zeta)\Big] +\left(1-t^{\delta-1}+t^{\delta-1}\zeta\right)\ln\lambda,
\end{split} \\ \nonumber
\begin{split}
&=\left(1-t^{\delta-1}+t^{\delta-1}\zeta\right)\Big[\ln\left(1-t^{\delta-1}\right)+\lambda_c\zeta-\frac{\lambda_c^2\zeta^2}{2}+\frac{\lambda_c^3\zeta^3}{3}+\dots\Big] \\
&\hspace{2cm}+t^{\delta-1}(1-\zeta)\Big[\ln\left(t^{\delta-1}\right)-\zeta-\frac{\zeta^2}{2}-\frac{\zeta^3}{3}-\dots)\Big] \\
&\hspace{2cm}+\left(1-t^{\delta-1}+t^{\delta-1}\zeta\right)\ln\lambda,
\end{split} \\ \label{eq:exponent}
&=c_0+c_1\zeta+c_2\zeta^2+c_3\zeta^3+\dots,
\end{align}
where the coefficients $c_j$ are evaluated as follows:
\begin{align*}
c_0&=\left(1-t^{\delta-1}\right)\ln\left(1-t^{\delta-1}\right)+t^{\delta-1}\ln t^{\delta-1}+\left(1-t^{\delta-1}\right)\ln\lambda=F\left(1-t^{\delta-1};\lambda\right); \\
c_1&=\left(1-t^{\delta-1}\right)\lambda_c+t^{\delta-1}\ln\left(1-t^{\delta-1}\right)-t^{\delta-1}-t^{\delta-1}\ln t^{\delta-1}+t^{\delta-1}\ln\lambda=t^{\delta-1}\ln\left(\frac{\lambda}{\lambda_c}\right); \\
c_n&=\left(1-t^{\delta-1}\right)\Big[\frac{(-1)^{n+1}\lambda_c^n}{n}\Big]+t^{\delta-1}\Big[\frac{(-1)^{n}\lambda_c^{n-1}}{n-1}\Big]+t^{\delta-1}\left(\frac{-1}{n}\right)-t^{\delta-1}\left(\frac{-1}{n-1}\right), \\
&=(-1)^n\Big[\frac{-t^{\delta-1}\lambda_c^{n-1}}{n}+\frac{t^{\delta-1}\lambda_c^{n-1}}{n-1}\Big]+\frac{t^{\delta-1}}{n(n-1)}, \\
&=\frac{t^{\delta-1}}{n(n-1)}\left(1-(-\lambda_c)^{n-1}\right)\sim\frac{t^{\delta-1}}{n(n-1)} \text{ for all $n\geq2$.}
\end{align*}
Using \eqref{eq:exponent} in \eqref{JB1-x}, we find:
\begin{align}
\nonumber J_{B1}(t;\la)&=t^{(\delta-1)/2}\int_0^{1-kt^{1-\delta}}(1-\zeta)^{-1/2}\,\re^{\ri t(c_0+c_1\zeta+c_2\zeta^2)}\,\re^{\ri t(c_3\zeta^3+c_4\zeta^4+\dots)}\,\mathrm{d}\zeta, \\
\label{JB1-with-IR}
&=\re^{\ri tc_0}t^{(\delta-1)/2}\int_0^{1-kt^{1-\delta}}(1-\zeta)^{-1/2}\,\re^{\ri t(c_1\zeta+c_2\zeta^2)}\,\mathrm{d}\zeta+I_R(t;\la),
\end{align}
where the remainder term $I_R$ is given by
\begin{align*}
I_R(t;\la)&=t^{(\delta-1)/2}\int_0^{1-kt^{1-\delta}}(1-\zeta)^{-1/2}\re^{\ri t(c_0+c_1\zeta+c_2\zeta^2)}\left(\re^{\ri t(c_3\zeta^3+c_4\zeta^4+\dots)}-1\right)\,\mathrm{d}\zeta, \\
&=t^{(\delta-1)/2}\int_0^{a}(1-\zeta)^{-1/2}\cO\left(\re^{\ri t(c_3\zeta^3+c_4\zeta^4+\dots)}-1\right)\,\mathrm{d}\zeta.
\end{align*}
The motivation for the substitution (\ref{a-defn}) now becomes clear: it simplifies the upper bound of the integral from $1-kt^{1-\delta}$ to simply $a$. 
To obtain the required result \eqref{JB1-asymptotics}, we will first prove that, under the assumption (\ref{JB1-a-assumption}), we have $I_R=\cO\big(t^{-\frac{1}{2}+\frac{3\delta}{2}}a^4\big)$.

Firstly, the exponent appearing in the integrand of $I_R$ is:
\begin{align*}
\ri t \sum_{n=3}^{\infty}c_n\zeta^n&=\ri t^{\delta}\sum_{n=3}^{\infty}\frac{\zeta^n}{n(n-1)}\left(1-(-\lambda_c)^{n-1}\right)=\ri t^{\delta}\bigg(\frac{\zeta^3}{6}+\cO\Bigg(\sum_{n=4}^{\infty}\frac{\zeta^n}{12}\left(1+\lambda_c^3\right)\bigg)\bigg), \\
&=\ri t^{\delta}\left(\frac{\zeta^3}{6}+\cO\left(\frac{\zeta^4}{6}(1-\zeta)^{-1}\right)\right)=\frac{\ri t^{\delta}\zeta^3}{6}+\cO\left(t^{\delta}\zeta^4\right).
\end{align*}
So $I_R$ itself can be estimated as follows:
\begin{align*}
I_R(t;\la)&=t^{(\delta-1)/2}\int_0^a(1-\zeta)^{-1/2}\cO\Bigg(\sum_{n=1}^{\infty}\frac{1}{n!}\left(\frac{\ri t^{\delta}\zeta^3}{6}+\cO\left(t^{\delta}\zeta^4\right)\right)^n\bigg)\,\mathrm{d}\zeta, \\
&=\cO\Bigg(t^{(\delta-1)/2}\int_0^{a}\left(\frac{\ri t^{\delta}\zeta^3}{6}+\cO\left(t^{\delta}\zeta^4,t^{2\delta}\zeta^6\right)\right)\,\mathrm{d}\zeta\bigg) \tas t\tendi,
\end{align*}
where we have used the second half of (\ref{JB1-a-assumption}), or equivalently $t^{\delta}a^3\ll1$, to ensure that the powers of $t^{\delta}\zeta^3$ in the exponential expansion do not increase to infinity. 
The second half of (\ref{JB1-a-assumption}) also implies that $a\ll1$, and so all higher-order terms, whether of the form $t^{\delta}\zeta^{3+K}$ or $\left(t^{\delta}\zeta^3\right)^K$ or a combination of both, are negligible compared to the leading term $t^{\delta}\zeta^3$. Thus $I_R$ satisfies
\begin{equation*}
I_R(t;\la)=\cO\Bigg(t^{(\delta-1)/2}\int_0^a\left(\frac{\ri t^{\delta}\zeta^3}{6}\right)\,\mathrm{d}\zeta\bigg)=\cO\left(t^{-\frac{1}{2}+\frac{3\delta}{2}}a^4\right).
\end{equation*}
Hence, equation (\ref{JB1-with-IR}) becomes:
\begin{align*}
J_{B1}(t;\la)&=\re^{\ri tc_0}t^{(\delta-1)/2}\int_0^{a}(1-\zeta)^{-1/2}\re^{\ri t(c_1\zeta+c_2\zeta^2)}\,\mathrm{d}\zeta+\cO\left(t^{-\frac{1}{2}+\frac{3\delta}{2}}a^4\right), \\
&=\re^{\ri tc_0}t^{(\delta-1)/2}\int_0^{a}\left(1+O(\zeta)\right)e^{\ri t(c_1\zeta+c_2\zeta^2)}\,\mathrm{d}\zeta+\cO\left(t^{-\frac{1}{2}+\frac{3\delta}{2}}a^4\right), \\
&=\re^{\ri tc_0}t^{(\delta-1)/2}\int_0^{a}\re^{\ri t(c_1\zeta+c_2\zeta^2)}\,\mathrm{d}\zeta+\cO\Bigg(t^{(\delta-1)/2}\int_0^{a}\zeta\mathrm{d}\zeta\bigg)+\cO\left(t^{-\frac{1}{2}+\frac{3\delta}{2}}a^4\right), \\
&=\re^{\ri tc_0}t^{(\delta-1)/2}\int_0^{a}\re^{\ri t(c_1\zeta+c_2\zeta^2)}\,\mathrm{d}\zeta+\cO\left(t^{-\frac{1}{2}+\frac{\delta}{2}}a^2\right)+\cO\left(t^{-\frac{1}{2}+\frac{3\delta}{2}}a^4\right).
\end{align*}
Since the first half of (\ref{JB1-a-assumption}) gives $t^{-\frac{1}{2}+\frac{\delta}{2}}a^2\ll t^{-\frac{1}{2}+\frac{3\delta}{2}}a^4$, we obtain
\begin{equation}
\label{JB1-asymptotics-1}
J_{B1}(t;\la)=
\re^{\ri tF(1-t^{\delta-1};\lambda)}t^{(\delta-1)/2}\int_0^a \exp\bigg[\ri t^{\delta}\left((\ln\tfrac{\lambda}{\lambda_c})\zeta+\tfrac{1}{2}(1+\lambda_c)\zeta^2\right)\bigg]\,\mathrm{d}\zeta
+\cO\left(t^{-\frac{1}{2}+\frac{3\delta}{2}}a^4\right).
\end{equation}

We now manipulate the integral on the right-hand side of \eqref{JB1-asymptotics-1} to obtain the desired result \eqref{JB1-asymptotics}.
Using the fact that
\beq\label{eq:trick1}
t\lambda_c=t^{\delta}(1+\lambda_c)
\eeq
(which follows from the definition (\ref{eq:la_c}) of $\lambda_c$), we find that the exponent of the integrand is
\begin{align*}
\ri t^{\delta}\bigg(\ln\left(\frac{\lambda}{\lambda_c}\right)\zeta+\frac{1}{2}(1+\lambda_c)\zeta^2\bigg)
&=\frac{\ri t^{\delta}(1+\lambda_c)}{2}\left(\zeta^2+2\frac{\ln(1+\Lambda)}{1+\la_c}\zeta\right),\\
&=\frac{\ri t\lambda_c}{2}\left(\zeta^2+2\sqrt{\frac{2}{\la_c t}}\omega\zeta\right),\\
&=\frac{\ri t\lambda_c}{2}\left(\zeta+\sqrt{\frac{2}{\la_c t}}\omega\right)^2- \ri\omega^2.
\end{align*}
So using the change of variables $\xi=\omega+\zeta \sqrt{t\la_c/2}$, the integral term in (\ref{JB1-asymptotics-1}) can be written as
\begin{align*}
&\exp\Big(\ri tF(1-t^{\delta-1};\lambda)-\ri\omega^2\Big)t^{(\delta-1)/2}\int_0^a \exp\left(\frac{\ri t\lambda_c}{2} \left(\zeta+\sqrt{\frac{2}{\la_c t}}\omega\right)^2\right)\,\mathrm{d}\zeta \\
&=\,\exp\Big(\ri tF(1-t^{\delta-1};\lambda)-\ri\omega^2\Big)t^{(\delta-1)/2}\sqrt{\frac{2}{\la_c t}}\int_{\omega}^{\omega+ a\sqrt{\la_c t/2}}\re^{\ri\xi^2}\,\mathrm{d}\xi,
\end{align*}
which becomes the main term in (\ref{JB1-asymptotics}) after using \eqref{eq:trick1}.
\end{proof}

\subsection{Step 4: Combining and unifying the asymptotics}

We now prove Theorem \ref{JB1+2-theorem} by combining the results of Lemmas \ref{JB2-theorem} and \ref{JB1-theorem}. This is the point where we need to be very precise about our choice of the splitting point $k$, or equivalently of the variable $a$ defined by \eqref{a-defn}, in order for these two results to be compatible. Note that as the order $m$ of the asymptotics increases, $a$ decreases by \eqref{a-assumption-sandwich}, and the error term from $J_{B1}$ becomes smaller and smaller in comparison to the series from $J_{B2}$. This makes sense, because when $a$ is very small, $k$ is very close to $t^{\delta-1}$, i.e. the integral in $J_{B2}$ is closer to the stationary point while the one in $J_{B1}$ is shorter, and so $J_{B1}$ contributes less to the final answer.

\

\bpf[Proof of Theorem \ref{JB1+2-theorem}]
When deriving the asymptotics for $J_{B2}$ in Lemma \ref{JB2-theorem}, we were still using a fairly general parameter $k$, required only to satisfy the condition (\ref{JB2-k-assumption}). But in Lemma \ref{JB1-theorem} we used a much more specific form of $k$, namely that given by (\ref{a-defn}) with $a$ satisfying (\ref{JB1-a-assumption}). 
In order to combine the asymptotics of $J_{B1}$ and $J_{B2}$, we first rewrite the results of Lemma \ref{JB2-theorem} with $k$ replaced by $t^{\delta-1}(1-a)$.

Firstly, we have \[D_-=\ln\left(\frac{t^{\delta-1}}{t^{\delta-1}(1-a)}\right)=-\ln(1-a)\sim a,\] and so the assumption (\ref{JB2-k-assumption}) can be rewritten as
\begin{equation}
\label{JB2-a-assumption}
t^{\epsilon\delta-\frac{m\delta}{2m+1}}\ll a\ll1.
\end{equation}
Note that taking $m=1$ would make (\ref{JB1-a-assumption}) and (\ref{JB2-a-assumption}) contradict each other, and so we must have
\begin{equation*}
t^{\epsilon\delta-\frac{m\delta}{2m+1}}\ll a\ll t^{-\delta/3}
\end{equation*}
for some $m\geq2$ and $\epsilon>0$. By Lemma \ref{JB2-theorem}, we then find that the estimate (\ref{JB2-asymptotics}) holds with
\begin{equation}
\label{Tj-estimate-a}
T_j(t;\la)=\cO\left((2j-1)!! t^{-\frac{1}{2}-\frac{(2j+1)\delta}{2}}a^{-2j-1}\right)
\end{equation}
for each $j$, and
\beqs
R_N(t;\la)=\cO\left((2N-1)!! t^{-\frac{1}{2}-\frac{(2N-1)\delta}{2}}(\ln t)^{(2N+1)/2}a^{-2N}\right).
\eeqs
In fact, we can improve \eqref{JB2-asymptotics} further. We assume that $m$ is minimal for (\ref{JB2-a-assumption}) to be valid, i.e. that
\begin{equation*}
t^{\epsilon\delta-\frac{m\delta}{2m+1}}\ll a\ll t^{\epsilon\delta-\frac{(m-1)\delta}{2m-1}},
\end{equation*}
which is implied by \eqref{a-assumption-sandwich}. As was discussed towards the end of the proof of Lemma \ref{JB2-theorem}, this means our estimate for $R_N$ is smaller than the one for $T_{N-m-1}$ but larger than the one for $T_{N-m}$. Thus the estimate \eqref{JB2-asymptotics} becomes:
\begin{align}\nonumber
\begin{split}
J_{B2}(t;\la)&=\sum_{j=1}^{N-m-1}T_j(t;\la)+\sum_{j=N-m}^{N-1}\cO\left((2j-1)!! t^{-\frac{1}{2}-\frac{(2j+1)\delta}{2}}a^{-2j-1}\right) \\
&\hspace{5cm}+\cO\left((2N-1)!! t^{-\frac{1}{2}-\frac{(2N-1)\delta}{2}}(\ln t)^{(2N+1)/2} a^{-2N}\right),
\end{split} \\ \label{eq:error}
&=\sum_{j=1}^{N-m-1}T_j(t;\la)+\cO\left((2N-1)!! t^{-\frac{1}{2}-\frac{(2N-1)\delta}{2}}(\ln t)^{(2N+1)/2}a^{-2N}\right),
\end{align}
where the error term is sufficiently small compared to the rest that it doesn't swallow up any of the remaining series.

Since we need to combine this result with the asymptotic formula (\ref{JB1-asymptotics}) for $J_{B1}$, we would like to ensure that the error term of $J_{B1}$, namely $\cO\big(t^{-\frac{1}{2}+\frac{3\delta}{2}}a^4\big)$, is \textit{also} sufficiently small so that it doesn't swallow up any of the terms in the above series. In other words, we require that our estimate for $T_j$ should be $\gg t^{-\frac{1}{2}+\frac{3\delta}{2}}a^4$ for all $j\leq N-m-1$. Checking this condition, we obtain
\begin{align*}
t^{-\frac{1}{2}-\frac{(2j+1)\delta}{2}}a^{-2j-1}\gg t^{-\frac{1}{2}+\frac{3\delta}{2}}a^4,
\quad\text{ which holds iff }\quad
a\ll t^{-(j+2)\delta/(2j+5)},
\end{align*}
which, by the assumption (\ref{a-assumption-sandwich}), is true provided that $j+2\leq m-1$. So we set $N-m-1=m-3$, i.e. $N=2m-2$. Now combining the two asymptotic expansions (\ref{JB1-asymptotics}) and (\ref{eq:error}) gives the required result (\ref{JB1+2-asymptotics}).
\end{proof}

\begin{remark}[Comparison of error terms] \label{rmk:error-comparison}
Comparing the error terms in \eqref{eq:error} and \eqref{JB1-asymptotics}, we find that, with our choice of $N=2m-2$,
\begin{align}\label{eq:error2}
t^{-\frac{1}{2}-\frac{(2N-1)\delta}{2}}(\ln t)^{(2N+1)/2}a^{-2N}\ll t^{-\frac{1}{2}+\frac{3\delta}{2}}a^4
\quad\text{ iff }\quad
a\gg t^{-\frac{\delta}{2}+\frac{\delta}{4m}}(\ln t)^{\frac{1}{2}-\frac{3}{8m}}.
\end{align}
If we assume $a$ takes the form $a=t^{-b\delta}$
for some constant $b$, then we can ignore the log terms. This is because the condition (\ref{a-assumption-sandwich}) can be rewritten as
\begin{equation*}
\frac{1}{2}-\frac{1}{4m-2}<b<\frac{1}{2}-\frac{1}{4m+2},
\end{equation*}
and the cutoff point \eqref{eq:error2} for which error term is dominant is 
\begin{equation}
\label{b-cutoff}
b=\frac{1}{2}-\frac{1}{4m},
\end{equation}
which is, in some sense, directly in the middle of the interval of possible values for $b$.
\end{remark}

\bco
The leading-order asymptotics of $J_B$ are given by
\begin{equation}\label{JB1+2-asymptotics-1}
\begin{split}
&J_B(t;\la)=\ri \,\exp\big(\ri tF(1-k;\lambda)\big)\,t^{-\frac{1}{2}-\frac{\delta}{2}}\left[\ln\left(\frac{1}{k}-1\right)+\ln\lambda\right]^{-1} \\
&\hspace{1.5cm}+\exp\Big(\ri tF(1-t^{\delta-1};\lambda)-
\ri\omega^2\Big)\,t^{-1/2}\sqrt{\frac{2}{1+\lambda_c}}\left(\int_{\omega}^{\omega + a \sqrt{\la_c t/2}}\re^{\ri\xi^2}\,\mathrm{d}\xi\right)+\cO\left(t^{\epsilon-\frac{1}{2}-\frac{\delta}{4}}\right),
\end{split}
\end{equation}
for any $\epsilon>0$, where $k=t^{\delta-1}(1-a)$, $a=t^{-7\delta/16}$, 
and $\omega$ is defined by \eqref{eq:omega}.

Moreover, \eqref{JB1+2-asymptotics-1} agrees with the asymptotics of $J_B$ found in Theorem \ref{thm:1} in the case $\sigma=1/2$.
\eco

\begin{proof}
We take $m=4$, the lowest possible value of $m$, and let $a=t^{-b\delta}$ as in Remark \ref{rmk:error-comparison}. By (\ref{b-cutoff}), the value of $a$ required to make both error terms in (\ref{JB1+2-asymptotics}) of comparable size is given by \[a=t^{-\frac{\delta}{2}+\frac{\delta}{4m}}=t^{-7\delta/16}.\] With this choice of $a$, the two error terms in (\ref{JB1+2-asymptotics}) are 
\beqs
\cO\left(t^{-\frac{1}{2}-\frac{\delta}{4}}(\ln t)^{(13/2)}\right)\quad\text{ and }\quad\cO\left(t^{-\frac{1}{2}-\frac{\delta}{4}}\right),
\eeqs
 which are both $\cO\big(t^{\epsilon-\frac{1}{2}-\frac{\delta}{4}}\big)$ as required.

Also, from the definition \eqref{eq:Tj} of $T_j$, we have
\beqs
T_1(t;\la)=\Bigg[\bigg(\frac{(1-z)^{-1/2}}{\ri t\frac{\partial F}{\partial z}}\bigg)\re^{\ri tF(z;\lambda)}\Bigg]_{z=1-k}^{\infty \re^{\ri \phi}}=\frac{-k^{-1/2}\re^{\ri tF(1-k;\lambda)}}{\ri t\frac{\partial F}{\partial z}\Big|_{z=1-k}} =\ri \re^{\ri tF(1-k;\lambda)}t^{-1}k^{-1/2}D^{-1}.
\eeqs
Then, using the definition \eqref{a-defn} of $a$ and the definition \eqref{d-defn} of $D$, we find
\begin{align*}
T_1(t;\la)&=\ri \re^{\ri tF(1-k;\lambda)}t^{-\frac{1}{2}-\frac{\delta}{2}}(1-a)^{-1/2}D^{-1} =\ri \re^{\ri tF(1-k;\lambda)}t^{-\frac{1}{2}-\frac{\delta}{2}}D^{-1}+\cO\left(t^{-\frac{1}{2}-\frac{\delta}{2}}a D_-^{-1}\right), \\
&=\ri \re^{\ri tF(1-k;\lambda)}t^{-\frac{1}{2}-\frac{\delta}{2}}\Big[\ln\left(\frac{1}{k}-1\right)+\ln\lambda\Big]^{-1}+\cO\left(t^{-\frac{1}{2}-\frac{\delta}{2}}\right),
\end{align*}
and this $\cO\big(t^{-\frac{1}{2}-\frac{\delta}{2}}\big)$ error term is absorbed by the $\cO\big(t^{-\frac{1}{2}-\frac{\delta}{4}}\big)$ error term in \eqref{JB1+2-asymptotics}.

It remains to show that (\ref{JB1+2-asymptotics-1}) agrees with the asymptotics in Theorem \ref{thm:1} in the case $\sigma=1/2$.
From \eqref{eq:JBtJB} and \eqref{eq:result1} we have
\begin{align}\label{eq:euan1}
J_B(t;\la)
=\sqrt{\frac{2}{(1+\la_c)t}}\exp\left(\frac{\ri tf_0(\La,\la_c)}{1+\lambda_c}-\ri \omega^2\right)\left(\int_{\omega}^{\infty \re^{\ri \pi/4}}\re^{\ri\xi^2}\,\mathrm{d}\xi\right) \,\bigg(1 + o(1)\bigg).
\end{align}

Using the definitions of $f_0$ \eqref{eq:f_0} and $\la_c$ \eqref{eq:la_c}, and some algebraic manipulation we find that
\begin{align*}
\frac{\ri tf_0(\La,\la_c)}{1+\lambda_c}-\ri\omega^2
=\ri tF(1-t^{\delta-1};\lambda),
\end{align*}
so \eqref{eq:euan1} becomes
\begin{equation}\label{eq:euan2}
J_B(t;\la)=\exp\Big(\ri tF(1-t^{\delta-1};\lambda)-\ri\omega^2\Big)t^{-1/2}\sqrt{\frac{2}{1+\la_c}}\left(\int_{\omega}^{\infty\re^{\ri\pi/4}}\re^{\ri\xi^2}\,\mathrm{d}\xi\right)\Big(1+o(1)\Big).
\end{equation}
From (\ref{JB1+2-asymptotics-1}) we have
\beq\label{eq:match1}
J_B(t;\la)=\exp\Big(\ri tF(1-t^{\delta-1};\lambda)-\ri\omega^2\Big)t^{-1/2}\sqrt{\frac{2}{1+\la_c}}\left(\int_{\omega}^{\infty\re^{\ri\pi/4}}\re^{\ri\xi^2}\,\mathrm{d}\xi\right) + r(t;\la),
\eeq
where the remainder $r(t;\la)$ is defined by
\begin{align}\nonumber
r(t;\la)&:= \ri \,\exp\big(\ri tF(1-k;\lambda)\big)\,t^{-\frac{1}{2}-\frac{\delta}{2}}\left[\ln\left(\frac{1}{k}-1\right)+\ln\lambda\right]^{-1} \\ 
& \qquad-\exp\Big(\ri tF(1-t^{\delta-1};\lambda)-\ri\omega^2\Big)t^{-1/2}\sqrt{\frac{2}{1+\la_c}}\left(\int_{\omega+ a\sqrt{\la_c t/2}}^{\infty\re^{\ri\pi/4}}\re^{\ri\xi^2}\,\mathrm{d}\xi\right).\label{eq:Remain}
\end{align}
If we can show that $r(t;\la)$ is little-o of the first term in \eqref{eq:match1} as $t\tendi$ (independently of $\la$), then the asymptotics \eqref{eq:match1} obtained from 
 Theorem \ref{JB1+2-theorem} are the same as \eqref{eq:euan2}, i.e., those obtained from Theorem \ref{thm:1}, and the proof is complete.
The asymptotics of the first term in \eqref{eq:match1} (which we can obtain from \eqref{eq:result2} in Theorem \ref{thm:1}) imply that it is sufficient to show that 
\beq\label{eq:Rasym}
r(t;\la) = 
\begin{cases} 
o(t^{-1/2}) &\text{ when }\omega=\cO(1), \\
o\big(t^{-1/2}t^{-\delta/2} (\log t)^{-1}\big) &\text{ when }\omega\tendi.
\end{cases}
\eeq
Now, from the definitions of $k$ \eqref{a-defn}, $\La$ \eqref{Lambda-defn}, and $\la_c$ \eqref{eq:la_c}, and the Taylor series \eqref{eq:log}, 
\begin{align*}
\ln\left(\frac{1}{k}-1\right) + \ln \la 
&= \ln \left( \left( \frac{t^{1-\delta}}{1-a}-1\right)\la_c (1+\La)\right),\\
& = \ln (1+\La) + \ln\left(\frac{1}{1-a}\right) + \ln \left(1 + \frac{a}{t^{1-\delta}-1}\right),\\
&= \ln (1+\La) + a +  \cO(a^2) +\frac{a}{t^{1-\delta}}\left(1+ \cO\left(\frac{1}{t^{1-\delta}}\right)\right) \quad\text{ as } t \tendi.
\end{align*}
Using these asymptotics in the first term of \eqref{eq:Remain}, and using the integration by parts \eqref{eq:9alt} in the second term, we find that 
\begin{align}\nonumber
r(t;\la)=
& \frac{\ri \,\exp\big(\ri tF(1-k;\lambda)\big)}{t^{1/2}t^{\delta/2} \Big[ \ln(1+\La) + a + \cO(a^2)\Big]}\\ \nonumber
&\hspace{1cm}
-\frac{\ri \exp\Big(\ri tF(1-t^{\delta-1};\lambda)-\ri\omega^2\Big) \re^{\ri (\omega + a\sqrt{\la_c t/2})^2}
}{
t^{1/2}(1+\la_c)^{1/2}\sqrt{2}\left(\omega+ a\sqrt{\frac{\la_c t}{2}}\right)
}
\left(1+ \cO\left(\left(\omega+ a\sqrt{\frac{\la_c t}{2}}\right)^{-2}\right)\right),\\ \nonumber
&=\frac{\ri \exp\big(\ri tF(1-t^{\delta-1};\lambda)\big)}{t^{1/2}t^{\delta/2}}\left(
\frac{
\exp\big( \ri t \big( F(1-k;\la) - F(1-t^{\delta-1};\la)\big)\big)
}{
\ln(1+\La) + a + \cO(a^2)
}
\right.\\ 
&\hspace{6cm}\left.
- \frac{
\exp\big( 2\ri \omega a \sqrt{\la_c t/2} + \ri a^2 \la_c t/2\big)
}{
\ln(1+\La) + a(1+ \la_c)
}
\big(1 + \cO(t^{-\delta/8})\big)
\right),\label{eq:Rbig}
\end{align}
where we have used both the definition of $\omega$ \eqref{eq:omega} and the equation \eqref{eq:trick1}. 

We now use \eqref{eq:omega} and \eqref{eq:trick1} to manipulate the final exponent in \eqref{eq:Rbig}:
\beqs
2 \omega a \sqrt{\frac{\la_c t}{2}} + \half a^2 \la_c t = t^{\delta} a \ln \left(\frac{\la}{\la_c}\right) + \half a^2 t^\delta (1+\la_c).
\eeqs
Therefore, the required asymptotics of $r(t;\la)$ \eqref{eq:Rasym} will follow from \eqref{eq:Rbig} if we can show that
\beq\label{eq:match2}
t \Big( F(1-k;\la) - F(1-t^{\delta-1};\la)\Big)= t^{\delta} a \ln \left(\frac{\la}{\la_c}\right) + \half a^2 t^\delta (1+\la_c) +\cO(t^{-\eps})
\eeq
for some $\eps>0$. The definition of $F(z;\la)$ \eqref{eq:2old} implies that
\beqs
F(z;\la)- F(w;\la) = (z-w)\bigg[ \ln\left(\frac{w}{1-w}\right) + \ln \la\bigg] + (1-z) \ln \left(\frac{1-z}{1-w}\right) + z\ln \left( \frac{z}{w}\right).
\eeqs
Using this, along with some algebraic manipulation, the equation \eqref{eq:trick1}, and the Taylor series \eqref{eq:log}, we find that 
the left-hand side of \eqref{eq:match2} equals 
\begin{align*}
&t^{\delta}a \ln\left(\frac{\la}{\la_c}\right) + t^{\delta}(1-a) \ln(1-a) + \big(t- t^{\delta}(1-a)\big) \ln(1+a\la_c),\\
&\qquad 
=t^{\delta}a \ln\left(\frac{\la}{\la_c}\right) + \half a^2 t^\delta(1+\la_c) + \cO(a^3 t^\delta),
\end{align*}
which is the right-hand side of \eqref{eq:match2}, since $a^3 t^\delta= t^{-5\delta/16}$; the proof is therefore complete.
\end{proof}


\end{document}